\documentclass[times,sort&compress,3p]{elsarticle}
\journal{}
\usepackage[labelfont=bf]{caption}

\usepackage{amsmath,amsfonts,amssymb,amsthm,booktabs,color,epsfig,graphicx,hyperref,url}

\usepackage{multirow}

\usepackage{bbm}
\usepackage{xpatch}

\newcommand{\lc}{\left[}
\newcommand{\rc}{\right]}

\newcommand{\Hb}{\mathbb{H}}
\newcommand{\ind}[1]{\mathbbm{1}_{\lrb{#1}}}
\newcommand{\inprod}[2]{\langle#1,#2\rangle_\Hb}
\newcommand{\lrp}[1]{\left(#1\right)}
\newcommand{\lrc}[1]{\left[#1\right]}
\newcommand{\lrb}[1]{\left\{#1\right\}}
\newcommand{\Esp}[1]{\mathbb{E}\lc #1\rc}

\theoremstyle{plain}
\newtheorem{theorem}{Theorem}

\newtheorem{lemma}{Lemma}

\theoremstyle{definition}

\newtheorem{remark}{Remark}

\xpatchcmd{\proof}{\itshape}{\prooflabelfont}{}{}
\newcommand{\prooflabelfont}{\bfseries}

\begin{document}

\begin{frontmatter}



\title{Testing the goodness-of-ﬁt of a functional autoregressive model} 
\author[1,3]{W. Gonz\'alez--Manteiga}
\author[2]{M.D. Ruiz--Medina \corref{mycorrespondingauthor}}
\author[1,3]{M. Febrero--Bande}
\affiliation[1]{organization={Univ. de Santiago de Compostela},
	addressline={Dpt. of Statistics, Math. Analysis and Optimization},
	city={Santiago de Compostela, (A Coru\~na)}, postcode={15781},country={Spain}
}
\affiliation[2]{organization={Univ. de Granada},
				addressline={Dpt. of Statistics and Op. Research},
				city={Granada}, postcode={18071},country={Spain}
}
\affiliation[3]{organization={Galician Centre for Mathematical Research and Technology},
	addressline={Campus Vida},
	city={Santiago de Compostela, (A Coru\~na)}, postcode={15782},country={Spain}
}

\cortext[mycorrespondingauthor]{Corresponding author. Email address: \url{mruiz@ugr.es}}




\begin{abstract}
The proposed Goodness--of--Fit  (GoF) test for checking the linear autocorrelation model in a functional time series is based on
an  empirical process, whose residual marks and covariate index set are in a separable Hilbert space $\mathbb{H}.$  A functional central limit theorem is derived providing the convergence of the  empirical process to a time--changed   Wiener process evaluated in a separable Hilbert space $\mathbb{H},$ with subordinator given by the marginal probability of the involved strictly stationary Autoregressive Hilbertian process  (AR$\mathbb{H}$(1) process).  The  large sample behavior of the test statistics is  obtained under  simple and composite  null hypotheses. The consistency of the test is addressed under simple null hypothesis.  The finite--sample performance of the testing procedure, under different families of alternatives, and random projection schemes, is illustrated  in  the Appendix.
\end{abstract}




\end{frontmatter}



\section{Introduction}
\label{secintro}
Functional time series  models have been extensively analyzed  since the'90s, supporting inference on stochastic processes.
The monograph by \cite{Bosq2000} constitutes a benchmarking in   asymptotic inference from   functional linear time series models in a state space framework. We also refer  the reader to \cite{Kokoszka10, Horvath2012a},  on weakly dependent functional times series analysis, and to \cite{Ferraty06} on nonparametric functional  statistics.

Current advances on statistical testing in functional time series cover different issues related to
  dependence structure, and   independence  (see, e.g., \cite{Hlavka21, Horvath2013, Zhang16};   separability, stationarity and linear correlation range   (see \cite{Constantinou18, Horvath2014, Kokoszka2013, Zhang15, Kokoszka2017});  normality and periodicity  (see
\cite{Gorecki18,Horman18}). Special attention has been paid to the issue of change point testing of the mean,   the variance, and  the autocorrelation operator  (see \cite{Horvath2012a}).

Specification and Goodness-of-Fit (GoF) tests are emerging areas in the Functional Data Analysis (FDA) literature (see \cite{Zhang16, Kim23, Kim24}, among others). In the context of the  functional  linear regression model, the asymptotic analysis of GoF tests in infinite--dimensional spaces  is currently an open challenge (see, e.g., \cite{Cuesta, Alvarez25, GPortugues21}, and the  references therein). Specifically, most of the proposals   are based on the application of random projection methodology
from  Theorem 4.1 in  \cite{Cuestalbertos}, typically  using bootstrap techniques  for calibration.
In the context of functional linear model with scalar response,  GoF asymptotic analysis, under simple and composite null hypothesis, is addressed from  randomly projected empirical processes in \cite{Cuesta}. The GoF asymptotics in the framework of the Functional Linear Model with Functional Response (FLMFR) entails several unsolvable  difficulties, also affecting  the usual asymptotic analysis based on random projection methodology (see, e.g., \cite{Alvarez25, GPortugues21}). In this context,  Theorem 4.1 in  \cite{Cuestalbertos} should be formulated  in terms of a non--degenerate bivariate infinite--dimensional Gaussian measure. This difficulty is overcoming in the present paper by considering inverse subordination of our test statistics, defined from a  Hilbert--valued ($\mathbb{H}$--valued) empirical process indexed by an $\mathbb{H}$--valued covariate. To this aim,  under strictly stationarity,  we introduce, in equation (\ref{eqinf}) below, the analogous to the concept of hitting time, in terms of an infinite--dimensional version of marginal quantile transformation, involving a tight infinite product Gaussian measure (see Theorem 1.2.1 in \cite{DaPrato2002}).

Our proposal extends the GoF asymptotic  analysis performed in
  \cite{KoulStute1999}, in the context of real--valued Markovian time series, to the functional time series framework under an Autoregressive Hilbertian process   of order one (AR$\mathbb{H}$(1) process) scenario. Note that this scenario corresponds  to the case of nondecreasing real--valued function $\psi$ being the identity in
  \cite{KoulStute1999}. In our functional  residual marked empirical process based design of the test statistics, we adopt the approach proposed in Lemma 1(d) in   \cite{Escanciano06}, reformulating it in an infinite--dimensional framework. Specifically,   we apply the fact that  any
 class  of Borel sets  in $\mathbb{H}$ containing a set  of  positive non--degenerate  tight infinite--dimensional Gaussian measure defines a separating class (see Theorem 4.1 in  \cite{Cuestalbertos}).  The separating class is then  indexed by the functions in the Reproducing Kernel Hilber Space (RKHS) of such an infinite--dimensional Gaussian measure, when the marginals of the AR$\mathbb{H}$(1) process satisfy Carleman condition. In particular, the index set of the separating class can be the functions in the   RKHS generated by the autocovariance operator of our AR$\mathbb{H}$(1) process (see Theorem 1.2.1 in \cite{DaPrato2002}).

The Central Limit Theorem  of the present paper  provides the asymptotic infinite--dimensional Gaussian  probability distribution of the marginals
 of the generalized functional residual marked empirical process,    under simple and composite null hypothesis.  Under both null hypothesis scenarios, the limiting  process is  identified  in law with time--changed $\mathbb{H}$--valued Brownian motion  via the derived
 Functional Central Limit Theorem (see  Theorem 2 in  \cite{Walk77}). Continuous Mapping Theorem then allows the asymptotic approximation  of the critical values of our infinite--dimensional version of Kolmogorov–Smirnov (K–S) test.  In practice, random projection  methodology can  be applied, and conditionally on a random direction, the weak convergence to real--valued  time--changed Brownian motion holds. The critical values are then asymptotically approximated  from the boundary crossing probabilities of  real--valued   Brownian motion on $[0,1].$  Our application of  random projection methodology leads to  an asymptotically distribution free (ADF) test  statistics.
 The consistency of the test is proved against  a broad class of alternatives, in the infinite--dimensional and conditionally one-dimensional formulations. Note that in the case of misspecified autocorrelation operator, the strong consistency results obtained in Chapter 8  in  \cite{Bosq2000}
 on projection estimation of the  autocorrelation operator  play a key role, in our derivation of the asymptotic analysis of our test statistics under composite  null hypothesis.

The finite--sample  properties of the test statistics are illustrated in the simulation study undertaken in the Appendix, under simple and composite null hypothesis, considering   two different alternative hypothesis scenarios. Specifically, we address the issues of detecting  nonlinearities in the functional covariates, and the problem of   the discrimination  between  two families of linear autocorrelation operators, in the context of invariant spherical functional time series (SP$\mathbb{H}$AR(1) processes). Under this second scenario, the  eigenfunctions of the autocovariance operator coincide with the eigenfunctions of the Laplace Beltrami operator in the sphere.
For relatively small functional sample sizes,  robust empirical  test sizes, and  good  empirical power  properties are observed  under simple $H_{0}.$   The impact of the misspecification level of the   autocorrelation operator   is also illustrated  under composite $H_{0}.$ As expected, when the misspecification level is lower (that is the case of known autocovariance eigenfunctions for  invariant SP$\mathbb{H}$AR(1) processes), better small sample size performance is observed  than in  the case of unknown eigenfunctions and eigenvalues.

The outline of the paper is the following. Some preliminary definitions and results  are provided in Section  \ref{sec1}. In particular, the formulation of the GoF testing procedure, and its implementation in practice are given in Section \ref{test}.
  The asymptotic distribution of the marginals, and the limiting process of the generalized functional residual marked empirical process are derived in Section \ref{AP}.
Consistency of the GoF test is addressed in  Section \ref{s4}.  The  misspecified autocorrelation operator scenario is analyzed  in Section \ref{s6}.
Final discussion is presented in Section \ref{fc}. Auxiliary information is  provided in the Appendix. Illustration of the finite sample performance of the proposed testing procedure is  provided in the simulation study undertaken in this Appendix.

\section{Preliminaries}\label{sec1}

Let $Y:=\{Y_{t},\ t\in \mathbb{Z}\}$  be a zero--mean  AR$\mathbb{H}$(1)  process  on the basic probability space $(\Omega,\mathcal{A},P)$  satisfying the state equation
\begin{equation}Y_{t}=\Gamma (Y_{t-1})+\varepsilon_{t},\quad t\in \mathbb{Z}.\label{eq1}\end{equation}
\noindent The autocorrelation operator  $\Gamma $ of $Y$ is assumed to be in the space  $\mathcal{L}(\mathbb{H})$  of bounded linear operators on $\mathbb{H}.$ The conditions ensuring the  existence of a unique stationary solution to equation (\ref{eq1}) are assumed (see, e.g.,  Chapter 3 in  \cite{Bosq2000}). In particular, we consider that $\left\|\Gamma \right\|_{\mathcal{L}(\mathbb{H})}<1.$
   The marginals of $Y$  are  in the space
 $\mathcal{L}^{2}_{\mathbb{H}}(\Omega,\mathcal{A},P)$ of zero--mean $\mathbb{H}$--valued random variables $X$  with $E\|X\|_{\mathbb{H}}^{2}<\infty.$ The autocovariance operator
   $C_{0}^{Y}=\mathbb{E}[Y_{0}\otimes Y_{0}]=\mathbb{E}[Y_{t}\otimes Y_{t}],$
$t\in \mathbb{Z},$  and the cross covariance operator
$C_{1}^{Y}=\mathbb{E}\left[Y_{0}\otimes Y_{1}\right]=  \mathbb{E}\left[Y_{t}\otimes Y_{t+1}\right],$ $t\in \mathbb{Z},$  are then nuclear operators.

The  innovation process $\varepsilon:=\{\varepsilon_{t},\ t\in \mathbb{Z}\}$   is assumed to be   $\mathbb{H}$--valued Strong White Noise ($\mathbb{H}$--SWN) (see Definition 3.1 in \cite{Bosq2000}). Thus,  $\varepsilon$ has  independent and identically distributed zero--mean  $\mathbb{H}$--valued random components,  with $C_{0}^{\varepsilon}:=\mathbb{E}\left[\varepsilon_{t} \otimes \varepsilon_{t}\right]=\mathbb{E}\left[\varepsilon_{0} \otimes \varepsilon_{0}\right],$ for all $t\in \mathbb{Z},$ having functional variance
$\mathbb{E}[\|\varepsilon_{t}\|^{2}_{\mathbb{H}}]=\mathbb{E}[\|\varepsilon_{0}\|^{2}_{
	\mathbb{H}}]=
	\sigma^{2}_{\varepsilon},$ which provides  the trace norm $\|C_{0}^{\varepsilon}\|_{L^{1}(\mathbb{H})}$ of the autocovariance operator $C_{0}^{\varepsilon}$ of $\varepsilon.$

The following assumption establishes the AR$\mathbb{H}$(1) setting,  where the results of the present paper are derived,  and where, in particular,  Carleman condition, $m_{n}:=\int \|x\|^{n}dP(x)<\infty ,$  $n\geq 1,$  and $\sum_{n\geq 1}m_{n}^{-1/n}=\infty,$ is satisfied by the respective marginals  of $Y$ and $\varepsilon,$  i.e., for $P= P_{Y_{0}}$  and $P=P_{\varepsilon_{1}}.$

\medskip

\noindent \textbf{Assumption A1}.  $Y$ is a  strictly stationary  AR$\mathbb{H}$(1) process, having $\mathbb{H}$--SWN innovation $\varepsilon $ such that  $\mathbb{E}\|\varepsilon_{1}\|^{4}_{\mathbb{H}}<\infty,$  with   marginals  $P_{Y_{0}}$  and $P_{\varepsilon_{1}}$  satisfying Carleman condition.

\medskip

Under model (\ref{eq1}), keeping in mind  Assumption A1, the following orthogonality condition holds \begin{equation}\mathbb{E}\left[(Y_{i}-\Gamma (Y_{i-1}))|Y_{i-1}\right]=\mathbb{E}[\varepsilon_{1}|Y_{0}]\underset{\mbox{a.s.}}{=}0\label{oc}\end{equation} \noindent  for $i\geq 1,$ with $\underset{\mbox{a.s.}}{=}$   denoting  the almost surely (a.s.)  equality.

\subsection{Separating class in the Borel $\sigma$--algebra $\mathcal{B}(\mathbb{H})$ of $\mathbb{H}$}

\noindent For a probability measure  $\mu $ in $\mathbb{H},$ its characteristic function   is given by  (see, e.g., Section 1.2.1 in \cite{DaPrato2002})
$$\phi_{\mu }(x) = \int_{\mathbb{H}}\exp\left( i\left\langle x,y\right\rangle_{\mathbb{H}}\right)\mu (dy),\quad x\in \mathbb{H}.$$

Let $P_{Y_{0}}$ be  the marginal probability measure induced by $Y_{0},$ and, for any Borel probability measure  $\nu $ in $\mathbb{H},$
consider  $$\mathcal{E}(P_{Y_{0}}, \nu):= \left\{ x\in \mathbb{H}:\ \phi_{P_{Y_{0}}}(tx)=\phi_{\nu }(tx),\ \forall t\in \mathbb{R}
\right\}.$$

\noindent In what follows,  $\mu_{Q} $ denotes  a centered non--degenerate Gaussian  measure  in $\mathbb{H}$ with positive trace autocovariance operator $Q.$
\begin{lemma}
\label{lem1}
Under  Assumption A1,  if  $\mu_{Q} \left(\mathcal{E}(P_{Y_{0}},\nu)\right)>0,$ then, $P_{Y_{0}}=\nu.$
\end{lemma}

\noindent Lemma \ref{lem1} follows straightforward from  Theorem 4.1 in  \cite{Cuestalbertos}.

\begin{remark}
\label{rsc}
From Lemma \ref{lem1}, for any Borel measure $\nu $ in $\mathbb{H},$ if $\nu(B)=P_{Y_{0}}(B),$ for
certain $B\in \mathcal{B}(\mathbb{H}),$ such that $\mu_{Q} (B)>0,$ then $P_{Y_{0}}=\nu,$  with $\mathcal{B}(\mathbb{H})$  denoting the Borel $\sigma$--algebra of $\mathbb{H}.$  Then, one can conclude that any
 class $\mathcal{V}(\mathbb{H})$ of Borel sets of $\mathbb{H}$ containing a set  of  positive $\mu_{Q}$--measure defines a separating class for the marginals of $Y.$
\end{remark}

In the introduction of a separating class of $\mathcal{B}(\mathbb{H}),$
  we consider the sets $E(x):=\prod_{j=1}^{\infty }(-\infty , x(\phi_{j})],$ $x\in \mathbb{H},$ for a given orthonormal basis $\left\{ \phi_{j},\ j\geq 1\right\}$ of $\mathbb{H},$ with $x(\phi_{j})=\left\langle  x,\phi_{j}\right\rangle_{\mathbb{H}},$
$j\geq 1.$ Note that \begin{eqnarray}P_{Y_{i}}\left( \prod_{j=1}^{\infty }(-\infty , x(\phi_{j})]\right)&:= &P\left(\{\omega \in \Omega ; \ \left\langle Y_{i}(\omega ),\phi_{j}\right\rangle_{\mathbb{H}}\leq \left\langle  x,\phi_{j}\right\rangle_{\mathbb{H}},\ j\geq 1\}\right)\nonumber\\
&= &\mathbb{E}\left[\ind{Y_{i}\in\prod_{j=1}^{\infty }(-\infty , \ x(\phi_{j})]}\right]=P_{Y_{0}}(E(x)),\quad  x\in \mathbb{H},\ i\geq 0,\end{eqnarray}\noindent   under Assumption A1.    In the next lemma,  the family of indicator functions $\mathcal{F}_{k-1}=\left\{\ind{Y_{k-1}\in\prod_{j=1}^{\infty }(-\infty , \ x(\phi_{j})]},\ x \in  Q^{1/2}(\mathbb{H})\right\}$ will play a crucial role in the characterization of the orthogonality condition (\ref{oc}) for each $k\geq 1,$ where
  $Q^{1/2}(\mathbb{H})$ denotes the RKHS generated by the kernel of a positive trace operator $Q.$ Lemma 1(d) in   \cite{Escanciano06}  is now  reformulated in our functional  context.

 \begin{lemma}
 \label{lem2}
 Under Assumption A1,  for each $k\geq 1,$
 \begin{eqnarray}
&& \hspace*{-0.75cm} \mathbb{E}[\varepsilon_{k}|Y_{k-1}]\underset{\mbox{a.s.}}{=}0 \ \mbox{if and only if} \
\mathbb{E}\left[\varepsilon_{k}\ind{Y_{k-1}\in E(x)}\right]
	=0,  \ \mbox{a.e.}  \ \mbox{in} \ x \in  Q^{1/2}(\mathbb{H}),\nonumber \\\label{occ2}
\end{eqnarray}
\noindent where $\ind{E^{k-1}(x)}:=\ind{Y_{k-1}\in E(x)}=\ind{Y_{k-1}\in\prod_{j=1}^{\infty }(-\infty , \ x(\phi_{j})]}$  is given here   in terms of the orthonormal basis $\left\{ \phi_{j},\ j\geq 1\right\}$
of  \    $\mathbb{H}$  satisfying
  $Q(\phi_{j})=\lambda_{j}(Q)\phi_{j},$ $j\geq 1,$ for a given  positive trace operator $Q$   on $\mathbb{H}.$
  \end{lemma}

\begin{proof}
Let $\widetilde{\mu}_{Q}=\prod_{j=1}^{\infty}\widetilde{\mu}_{0,\lambda_{j}(Q)}$ be the infinite product Gaussian measure on $(\mathbb{R}^{\infty},\mathcal{B}(\mathbb{R}^{\infty})),$ identified in $l^{2}$ sense with an infinite--dimensional  Gaussian measure $\mu_{Q}$ on $(\mathbb{H},\mathcal{B}(\mathbb{H}))$ with positive trace autocovariance operator $Q,$ satisfying $Q(\phi_{j})=\lambda_{j}(Q)\phi_{j},$
 $j\geq 1$
(see Theorem 1.2.1 in \cite{DaPrato2002}). Here, for each $j\geq 1,$ $\widetilde{\mu}_{0,\lambda_{j}(Q)}$ denotes a one--dimensional Gaussian measure with zero mean and variance $\lambda_{j}(Q).$
  For $x\in  Q^{1/2}(\mathbb{H}),$ $\sum_{j=1}^{\infty}\left(1-\Phi \left(\frac{\left\langle  x,\phi_{j}\right\rangle_{\mathbb{H}},}{\lambda_{j}(Q)}\right)\right)<\infty.$
  Hence,   $\prod_{j=1}^{\infty }\Phi \left(\frac{\left\langle  x,\phi_{j}\right\rangle_{\mathbb{H}}}{\lambda_{j}(Q)}\right)>0,$  where $\Phi $ denotes the standard normal cumulative distribution function. Thus,
  \begin{eqnarray}&&\hspace*{-1cm} \mu_{Q}\left(E(x)\right)=\widetilde{\mu}_{Q}\left( \prod_{j=1}^{\infty}(-\infty, x(\phi_{j})]\right)=\prod_{j=1}^{\infty }\Phi \left(\frac{\left\langle  x,\phi_{j}\right\rangle_{\mathbb{H}}}{\lambda_{j}(Q)}\right)>0,\label{pm}\end{eqnarray}
 \noindent with, as before,   $x(\phi_{j})=\left\langle  x,\phi_{j}\right\rangle_{\mathbb{H}},$ $j\geq 1.$
 From Lemma  \ref{lem1},  Remark  \ref{rsc} and equation  (\ref{pm}),  $\{ E(x),\ x\in Q^{1/2}(\mathbb{H})\}$  defines a separating class of Borel sets of $\mathbb{H}. $
 This separating condition is sufficient for  the class of functions $\mathcal{F}_{k-1}=\left\{\ind{ E^{k-1}(x)},\ x\in Q^{1/2}(\mathbb{H})\right\}$ to satisfy  (\ref{occ2}).
\end{proof}

 \medskip

For each $x\in \mathbb{H},$ the next lemma provides   the martingale difference  property of the  $\mathbb{H}$--valued    sequence $X(x):=\{X_{i}(x),\ i\geq 1\}=\left\{\varepsilon_{i}\ind{E^{i-1}(x)}),\ i\geq 1\right\}.$

\begin{lemma}
	\label{lem1a} For each $x\in \mathbb{H},$ the sequence
	$$X(x)=\{X_{i}(x),\ i\geq 1\}=\{(Y_{i}-\Gamma (Y_{i-1}))\ind{E^{(i-1)}(x)},\ i\geq 1\}$$  \noindent  is an $\mathbb{H}$--valued martingale difference  with respect to  the filtration  $\mathcal{M}_{0}^{Y}\subset \mathcal{M}_{1}^{Y}\dots \subset \mathcal{M}_{n}^{Y}\subset \dots,$  where    \linebreak $\mathcal{M}_{i-1}^{Y}=\sigma (Y_{t},\ t\leq i-1),$
	for $ i\geq 1.$
	
\end{lemma}
\begin{proof}
For each $i\geq 1,$ and $x\in\mathbb{H},$
	\begin{eqnarray}&&\hspace{-1cm}\mathbb{E}\left[X_{i}(x)|\mathcal{M}_{i-1}^{Y}\right]:= \mathbb{E}^{\mathcal{M}_{i-1}^{Y}}\left[X_{i}(x)\right]=
		\mathbb{E}^{\mathcal{M}_{i-1}^{Y}}\left[\left(Y_{i}-\Gamma (Y_{i-1})\right)\ind{ \omega \in \Omega;\ \left\langle Y_{i-1}(\omega ),\phi_{j}\right\rangle_{\mathbb{H}}\leq \left\langle  x,\phi_{j}\right\rangle_{\mathbb{H}},\ j\geq 1}\right]\nonumber\\
		&&=\ind{ \omega \in \Omega;\ \left\langle Y_{i-1}(\omega ),\phi_{j}\right\rangle_{\mathbb{H}}\leq \left\langle  x,\phi_{j}\right\rangle_{\mathbb{H}},\ j\geq 1 }\mathbb{E}^{\mathcal{M}_{i-1}^{Y}}\left[Y_{i}-\Gamma (Y_{i-1})\right]
		= \ind{ \omega \in \Omega;\ \left\langle Y_{i-1}(\omega ),\phi_{j}\right\rangle_{\mathbb{H}}\leq \left\langle  x,\phi_{j}\right\rangle_{\mathbb{H}},\ j\geq 1}
		\mathbb{E}^{\mathcal{M}_{i-1}^{Y}}\left[\varepsilon_{i}\right]=0,\ \mbox{a.s.}\  \ i\geq 1,\nonumber
	\end{eqnarray}
	\noindent in view of the definition of the innovation process $\varepsilon $ as $\varepsilon_{n}= Y_{n}-\Pi^{\mathcal{M}_{n-1}}(Y_{n})=Y_{n}-\Gamma (Y_{n-1}),$ for every $n\geq 1,$ where
	$\Pi^{\mathcal{M}_{n-1}}$ is the orthogonal projector into $\mathcal{M}_{n-1}$  (see equation (2.55) and  (3.4) in  \cite{Bosq2000}).
\end{proof}

\subsection{Testing the autocorrelation model}
\label{test}
For a given   $\Gamma_{0}\in \mathcal{L}(\mathbb{H}),$  consider the testing problem:
\begin{eqnarray}H_{0}:\Gamma &=&\Gamma_{0},\nonumber\\
	H_{1}:\Gamma &\neq &\Gamma_{0}.\nonumber
\end{eqnarray}
Our formulation of K–S test is based on Lemma \ref{lem2}, and the
  $\mathbb{H}$--valued empirical process $V_{n}:=\left\{ V_{n}(x),\  x\in \mathbb{H}\right\}$
\begin{eqnarray}V_{n}(x)&=&\frac{1}{\sqrt{n}}\sum_{i=1}^{n}X_{i}(x)=
	\frac{1}{\sqrt{n}}\sum_{i=1}^{n}(Y_{i}-\Gamma_{0} (Y_{i-1}))\ind{E^{i-1}(x)}=\frac{1}{\sqrt{n}}\sum_{i=1}^{n}\varepsilon_{i}(\Gamma_{0})\ind{E^{i-1}(x)},\quad x\in \mathbb{H}.\label{efp}
\end{eqnarray}

Theorem \ref{th3}  below establishes the  weak convergence of $V_{n}$ to  an $\mathbb{H}$--valued  Gaussian process $ W_{\infty}:=\left\{ W_{\infty }(x),\ x\in \mathbb{H}\right\}$ with
 covariance operator $C_{\min(x,y)}^{W_{\infty}}$ defined,  for every  $x,y\in \mathbb{H},$ as
     \begin{equation}\hspace*{-0.5cm}C_{\min(x,y)}^{W_{\infty}}:= \mathbb{E}\left[ W_{\infty}(x)\otimes W_{\infty}(y)\right]=C_{0}^{\varepsilon}P_{Y_{0}}\left[\prod_{j=1}^{\infty }\left(-\infty ,\min\{ x(\phi_{j}),y(\phi_{j})\}\right]\right].\label{cov2}\end{equation}

  \noindent Here,
 \begin{eqnarray} P_{Y_{0}}\left[\prod_{j=1}^{\infty }(-\infty ,\min\{ x(\phi_{j}),y(\phi_{j})\}]\right]  &:=&P\left( \omega \in \Omega ; \ \left\langle Y_{0}(\omega ),\phi_{j}\right\rangle_{\mathbb{H}}\leq \min\left(\left\langle  x,\phi_{j}\right\rangle_{\mathbb{H}},\left\langle  y,\phi_{j}\right\rangle_{\mathbb{H}}\right),\ j\geq 1\right)
 \nonumber\\ \hspace*{1cm} &=&
 \mathbb{E}\left[\ind{Y_{0}\in \prod_{j=1}^{\infty}\left(-\infty, \min\{x(\phi_{j}),y(\phi_{j})\}\right]}\right].
 \label{eqmin}\end{eqnarray}

From equations (\ref{cov2})--(\ref{eqmin}), process $W_{\infty }$ can be identified in law  with time--changed $\mathbb{H}$--valued Wiener process with subordinator  \begin{equation}\left\{P_{Y_{0}}(x):=P_{Y_{0}}\left[E(x)\right]=P_{Y_{0}}\left[\prod_{j=1}^{\infty }(-\infty , x(\phi_{j})]\right],\ x\in \mathbb{H}\right\}.\label{eqsub}
\end{equation}
\noindent  That is, $W_{\infty }$ can be generated from $\mathbb{H}$--valued  Wiener process on $[0,1],$ by inverse subordination
via $P_{Y_{0}}^{-1},$ with

\begin{eqnarray} P_{Y_{0}}^{-1}(t):= x_{0}\ \mbox{such that} \   \mu_{C_{0}^{Y}}\left( \prod_{j=1}^{\infty }[x_{0}(\phi_{j}),x(\phi_{j})]\right)& > & 0,\nonumber\\
 \forall x\in [C_{0}^{Y}]^{1/2}(\mathbb{H}); \ P_{Y_{0}}\left( \prod_{j=1}^{\infty} \left(-\infty, x(\phi_{j})\right]\right) &\geq & t,\ t\in [0,1],\label{eqinf}
\end{eqnarray}

\noindent where $\left\{\phi_{j},\ j\geq 1\right\}$ now denotes the orthonormal basis of $\mathbb{H},$ given by the eigenfunctions of the trace autocovariance operator $C_{0}^{Y}.$ Note that in our $\mathbb{H}$--valued context, infinite--dimensional tight Gaussian measures play the role of the uniform measure in finite--dimensions (see  Theorem 4.1 in  \cite{Cuestalbertos}).

Applying Theorem  \ref{th3}, and Continuous Mapping Theorem, we arrive at the following asymptotic approximation of K–S test in infinite dimensions:
\begin{equation}\sup_{t\in [0,1]} \left\|V_{n}(P_{Y_{0}}^{-1}(t))\right\|_{\mathbb{H}}\underset{D}{\simeq }\sup_{x\in \mathbb{H}}\left\|W_{C_{0}^{\varepsilon}}\circ P_{Y_{0}}(x)\right\|_{\mathbb{H}}=\sup_{t\in [0,1]}\left\|W_{C_{0}^{\varepsilon}}(t)\right\|_{\mathbb{H}},\label{cmthbb}
\end{equation}
\noindent  as $n\to \infty,$ where $W_{C_{0}^{\varepsilon}}=\left\{ W_{C_{0}^{\varepsilon}}(t),\  t\in [0,1]\right\}$ denotes an $\mathbb{H}$--valued Wiener process on the interval  $[0,1],$ with autocovariance operator $C_{0}^{\varepsilon},$ in the sense introduced in Definition 2 of \cite{Dedecker}. This process induces  a probability measure on $C_{\mathbb{H}}([0,1]),$   the separable  Banach space,  under the supremum norm $\|g\|_{\infty}=\sup_{t\in [0,1]}\|g(t)\|_{\mathbb{H}},$ of $\mathbb{H}$--valued continuous functions on $[0,1]$ with respect to $\mathbb{H}$ norm. From (\ref{cmthbb}),
the critical values of the test based on $\sup_{t\in [0,1]} \left\|V_{n}(P_{Y_{0}}^{-1}(t))\right\|_{\mathbb{H}}$  can be approximated from the boundary crossing probabilities of  $W_{C_{0}^{\varepsilon}}.$

  The implementation in practice of this testing criterion is complicated. Particularly, we are interested on  arriving at an ADF test, which will be now introduced via random projection methodology. Then, we consider the test statistics
  \begin{eqnarray}&&\mathcal{T}(\mathbf{h})= \sup_{t\in [0,1]} \left|s_{n}^{-1}(\mathbf{h})\left\langle V_{n}(P_{Y_{0}}^{-1}(t)), \mathbf{h}\right\rangle_{\mathbb{H}}\right|= 	 \sup_{t\in [0,1]} \left|\frac{s_{n}^{-1}(\mathbf{h})}{\sqrt{n}}\sum_{i=1}^{n}\left\langle \varepsilon_{i}(\Gamma_{0}),\mathbf{h}\right\rangle_{\mathbb{H}}\ind{E^{i-1}(P_{Y_{0}}^{-1}(t))}\right|,
\label{statistic}
\end{eqnarray}
  \noindent conditionally to  the functional value  $\mathbf{h}\in \mathbb{H}$ of a Gaussian distributed functional random  variable
  inducing measure $\mu_{C_0^{\epsilon}}.$  Here,
\begin{eqnarray}
	s_{n}^{-1}(\mathbf{h})&=&\left(\frac{1}{n}\sum_{i=1}^{n}\left[\left\langle (Y_{i}-\Gamma_{0} (Y_{i-1})),\mathbf{h}\right\rangle_{\mathbb{H}} \right]^{2} \right)^{-1/2}= \left(\frac{1}{n}\sum_{i=1}^{n}\left[\left\langle \varepsilon_{i}(\Gamma_{0}), \mathbf{h}\right\rangle_{\mathbb{H}}\right]^{2}\right)^{-1/2}.
	\label{rpci}
\end{eqnarray}

\noindent Note that, from Theorem  \ref{th3} below,  under $H_{0},$  $$Z_{\mathbf{h}}=\left\{ Z_{\mathbf{h}}(t),\  t\in [0,1]\right\}=\left\{\frac{s_{n}^{-1}(\mathbf{h})}{\sqrt{n}}\sum_{i=1}^{n}\left\langle \varepsilon_{i}(\Gamma_{0}),\mathbf{h}\right\rangle_{\mathbb{H}}\ind{E^{i-1}(P_{Y_{0}}^{-1}(t))},\ t\in [0,1]\right\}$$
\noindent conditionally converges in law to a Gaussian process
on the interval $[0,1],$ with covariance function,  given by,  for   $t,s\in [0,1],$

\begin{eqnarray} C_{Z_{\mathbf{h}}}(t,s)&=& \frac{C_{0}^{\varepsilon}(\mathbf{h})(\mathbf{h})}{C_{0}^{\varepsilon}(\mathbf{h})(\mathbf{h})}
P_{Y_{0}}\left[ \prod_{j=1}^{\infty }\left(-\infty , \min\left\{[P_{Y_{0}}^{-1}(t)](\phi_{j}),[P_{Y_{0}}^{-1}(s)](\phi_{j})\right\}\right]\right]\nonumber\\
&=&P_{Y_{0}}\left[ \prod_{j=1}^{\infty }\left(-\infty , \min\left\{[P_{Y_{0}}^{-1}(t)](\phi_{j}),[P_{Y_{0}}^{-1}(s)](\phi_{j})\right\}\right]\right].\nonumber\end{eqnarray}

\noindent This process can be  identified in law with  real--valued time--changed Brownian motion with inverse subordinator $P_{Y_{0}}^{-1}.$ Thus,
as $n\to \infty,$
\begin{equation}\sup_{t\in [0,1]} \left|s_{n}^{-1}(\mathbf{h})\left\langle V_{n}(P_{Y_{0}}^{-1}(t)), \mathbf{h}\right\rangle_{\mathbb{H}}\right| \underset{D}{\simeq }\sup_{t\in [0,1]}\left|W(t)\right|,\label{cmth}
\end{equation}
\noindent where $W$ denotes real--valued Brownian motion on $[0,1].$ By Reflection Principle,   the identity $P\left[\sup_{t\in [0,1]}W(t)\leq a\right]=2\Phi(a)-1$ then allows the approximation of  the critical values of this random projected based test statistics.  As before, $\Phi$ denotes the  standard normal cumulative probability distribution function.
	
Under conditions of  Lemmas  \ref{lem1}--\ref{lem2},   since, under  Assumption A1,   the  marginals of $\varepsilon $ also  satisfy Carleman condition, the procedure can be summarized as  follows:  The null hypothesis $H_{0}$
holds if and only if, for $\mathbf{h}\in B,$ with $B\in \mathcal{B}(\mathbb{H})$ such that  $\mu_{C_{0}^{\varepsilon}}(B)>0,$ $\mathcal{T}(\mathbf{h})$ does not exceed a  critical value of the probability distribution of the supremum norm of   Brownian motion on  $[0,1].$  Otherwise,
$\mathcal{T}(\mathbf{h})$  exceeds a  critical value of the probability distribution of the supremum norm of   Brownian motion  on  $[0,1]$ $\mu_{C_{0}^{\varepsilon}}$--a.s. in $\mathbf{h}\in  \mathbb{H},$ which is equivalent to
 $H_{0}$ fails with probability one in view of Theorem 4.1 in  \cite{Cuestalbertos}.

\begin{remark}
In practice,  $\mathbf{h}$ can be generated from an infinite--dimensional  Gaussian measure $\mu_{C_{0}^{\varepsilon}}.$ Thus,  $\mathbf{h}$ can be   generated from the identity
\begin{equation}
\mathbf{h}=\sum_{j\geq 1}\sqrt{\lambda_{j}(C_{0}^{\varepsilon})}\epsilon_{j}\widetilde{\phi}_{j},\label{KLE}
\end{equation}
\noindent where  $C_{0}^{\varepsilon}(\widetilde{\phi}_{j})=\lambda_{j}(C_{0}^{\varepsilon})\widetilde{\phi}_{j},$ $j\geq 1,$ and $\{\epsilon_{j},\  j\geq 1\}$ is a sequence of independent and identically distributed standard Gaussian random variables.

Note  that equation (\ref{KLE}) can be approximated in a consistent way from the results derived in  Chapters 4 and 8 of \cite{Bosq2000}.  Given the $\mathbb{H}$--SWN property of $\varepsilon,$  alternative strongly consistent approximations of the eigenvalues  of $C_{0}^{\varepsilon}$ follow from the strong  law of large numbers. Specifically,  for the case where    $\left\{\widetilde{\phi}_{j},\ j\geq 1\right\}$ is  known, one can consider

 \begin{eqnarray}&&
	\widehat{\lambda}_{j}(C_{0}^{\varepsilon})=\frac{1}{n}\sum_{i=1}^{n}\left[ (Y_{i}-\Gamma (Y_{i-1}))(\widetilde{\phi}_{j})\right]^{2} = \frac{1}{n}\sum_{i=1}^{n}[\varepsilon_{i}(\widetilde{\phi}_{j})]^{2}.
\end{eqnarray}
While,  in the case of  $\left\{\widetilde{\phi}_{j},\ j\geq 1\right\}$ being unknown, the empirical complete orthonormal eigenfunction system $\left\{\widetilde{\phi}_{j,n},\ j\geq 1\right\},$ and the corresponding empirical eigenvalues $\left\{\lambda_{j,n}(C_{0}^{\varepsilon}),\ j=1,\dots ,n\right\}$   of $C_{0}^{\varepsilon}$ can be considered.  They satisfy the following identity $\frac{1}{n}\sum_{i=1}^{n}\varepsilon_{i}\otimes \varepsilon_{i}(\widetilde{\phi}_{j,n})=
\lambda_{j,n}(C_{0}^{\varepsilon})\widetilde{\phi}_{j,n},$ $j=1,\dots,n,$ with $\lambda_{n+k,n}(C_{0}^{\varepsilon})=0,$ for $k\geq 1.$
\end{remark}

\section{Asymptotic probability distribution  under simple $H_0$}
\label{AP}
The main asymptotic results are derived in this section. Specifically, Theorem \ref{th1} in Section \ref{s2} provides, for each $x\in \mathbb{H},$  the convergence in distribution of the marginal $V_{n}(x)$ to $V_{\infty }(x)\sim \mathcal{N}(0,C_{0}^{\varepsilon}P_{Y_{0}}(E(x))).$ The random variable $V_{\infty }(x)$ is  distributed as a  zero--mean $\mathbb{H}$--valued Gaussian random variable  with autocovariance operator $C_{0}^{\varepsilon}P_{Y_{0}}(E(x)).$  The weak convergence of  $V_{n}$   to time--changed $\mathbb{H}$--valued  Wiener process $W_{C_{0}^{\varepsilon}}\circ P_{Y_{0}}=\left\{W_{C_{0}^{\varepsilon}}\circ P_{Y_{0}}(x),\ x\in \mathbb{H}\right\},$ with subordinator   $P_{Y_{0}}(x)$ introduced in equation (\ref{eqsub}), and covariance operator given in equation (\ref{cov2}), is proved in
Theorem \ref{th3} in Section \ref{s3}.

Through this section we omit the dependence of the functional marks $\{\varepsilon_{j}\}$  on the totally specified autocorrelation operator $\Gamma_{0}$ under $H_{0}.$

\medskip

The next assumption is of technical nature, and  plays a crucial role in the derivation of the results of this paper.

\medskip

\noindent \textbf{Assumption A2}.  Assume  that $Y_{0}$ is independent of $ \varepsilon_{i},$ for all  $i\geq 1.$

\begin{remark}
The following AR$\mathbb{H}$(1) formula motivates the introduction of Assumption A2 above, for the derivation of Theorems \ref{th1} and \ref{th3} below (see  (3.11) in \cite{Bosq2000})
\begin{eqnarray}
	Y_{i-1}=\sum_{t=0}^{i-2}\Gamma^{t}(\varepsilon_{i-1-t})+\Gamma^{i-1}(Y_{0}),\quad i\geq 1.
	\label{eq311b2000}
\end{eqnarray}
\end{remark}

\subsection{Central Limit Theorem}
\label{s2} Theorem \ref{th1} is now derived by applying  Theorem 2.16  in \cite{Bosq2000} to $\left\{ X_{i}(x),\ i\geq 1 \right\}$ in Lemma \ref{lem1a}, for each $x\in \mathbb{H}.$

\begin{theorem}\label{th1} Let $Y$ be a   centered  \mbox{\emph{AR}}$\mathbb{H}$\mbox{\emph{(1)}} process, satisfying
	Assumptions A1--A2. Then,
	for each $x\in \mathbb{H},$  the following convergence holds:  \begin{eqnarray}&&V_{n}(x)\to_{D}V_{\infty }(x)\sim\mathcal{N}(0,C_{x}),\quad n\to \infty,\nonumber \end{eqnarray}
\noindent with $\to_{D}$ denoting the convergence in distribution, and,  for each $x\in \mathbb{H},$ the autocovariance operator $C_{x}$ of the limiting Gaussian random variable $V_{\infty }(x)$ is given by $C_{x}=C_{0}^{\varepsilon}P_{Y_{0}}(E(x))=
	C_{0}^{\varepsilon}P_{Y_{0}}(x).$  Thus, as
	$\left\langle  x,\phi_{j}\right\rangle_{\mathbb{H}}\to \infty,$
	 $j\geq 1,$
	$$V_{n}(x)\to_{D}\mathcal{Z}\sim
	\mathcal{N}(0,C_{0}^{\varepsilon}),\quad n\to \infty.$$	
\end{theorem}
\begin{proof}
The proof of this result consists of the verification of conditions (2.59)--(2.61)  in Theorem 2.16  in \cite{Bosq2000}. Specifically, condition (2.59) follows from
Vitali Convergence  Theorem by proving convergence in probability to zero, and uniform integrability of $\left\{ X_{i}(x),\ i\geq 1 \right\}.$  Condition (2.61) provides tightness, and condition (2.60) requires the application of  (2.36) in  Corollary 2.3 in \cite{Bosq2000}.

Let us first consider  condition (2.59) in Theorem 2.16. For each $x\in \mathbb{H},$ we have to prove
\begin{equation}n^{-1/2}\mathbb{E}\left( \max_{1\leq i\leq n}\left\|X_{i}(x)\right\|_{\mathbb{H}}\right)\to 0,\quad n\to \infty.\label{convmean}
\end{equation}

Applying Chebyshev inequality, $\mathbb{H}$--SWN property of $\varepsilon ,$ and strictly stationarity of $Y$ under Assumption A1,
\begin{eqnarray}&& \hspace{-1cm}P\left(\max_{1\leq i\leq n}\left\|X_{i}(x)\right\|_{\mathbb{H}}>\sqrt{n}\eta \right)\leq
\sum_{i=1}^{n}P\left(\left\|X_{i}(x)\right\|_{\mathbb{H}}>\sqrt{n}\eta \right)\nonumber\\ &&
\leq
\frac{1}{n\eta^{2}}\sum_{i=1}^{n}\Esp{\|X_{i}(x)\|_{\mathbb{H}}^{2}\ind{\|X_{i}(x)\|_{\mathbb{H}}>\sqrt{n}\eta}}
=\frac{1}{\eta^{2}}\Esp{\|X_{1}(x)\|_{\mathbb{H}}^{2}\ind{\|X_{1}(x)\|_{\mathbb{H}}>\sqrt{n}\eta}}.
	\nonumber\\
\label{eqeventa}
\end{eqnarray}

The Dominated Convergence Theorem yields convergence to zero of  $\Esp{\|X_{1}(x)\|_{\mathbb{H}}^{2}\ind{\|X_{1}(x)\|_{\mathbb{H}}>\sqrt{n}\eta}},$ leading, from equation (\ref{eqeventa}), to  \begin{equation}P\left(\max_{1\leq i\leq n}\left\|X_{i}(x)\right\|_{\mathbb{H}}>\sqrt{n}\eta \right)\to 0,\quad n\to \infty.\label{convmean2}
\end{equation}

Under Assumption A1,

\begin{eqnarray}
	&&\Esp{\max_{1\leq i\leq n}\left\|\frac{X_{i}(x)}{\sqrt{n}}\right\|_{\mathbb{H}}^{2}}\leq  \Esp{\left\|X_{1}(x)\right\|^{2}_{\mathbb{H}}}<\infty
\nonumber \\ \label{eqmax}\end{eqnarray}
\noindent ensuring uniform integrability. Thus,   from equations (\ref{convmean2}) and (\ref{eqmax}),
equation (\ref{convmean})  follows from Vitali Convergence  Theorem.

Let now prove, for each $x\in \mathbb{H},$
$$\lim_{N\to \infty} \lim_{n\to \infty}
P\left[\sum_{i=1}^{n}r_{N}^{2}\left(\frac{X_{i}(x)}{\sqrt{n}}\right)>\eta\right]=0,\quad \eta >0,$$
\noindent corresponding to condition  (2.61) in Theorem 2.16 in \cite{Bosq2000},  where
$r_{N}^{2}\left(\frac{X_{i}(x)}{\sqrt{n}}\right)=\sum_{l=N}^{\infty}
		\left\langle \frac{X_{i}(x)}{\sqrt{n}},\phi_{l}\right\rangle^{2}_{\mathbb{H}}.$
Specifically, for any $n,$

\begin{eqnarray}P\lrc{\sum_{i=1}^{n}r_{N}^{2}\lrp{\frac{X_{i}(x)}{\sqrt{n}}}>\eta}&\leq & \sum_{i=1}^{n}P\lrc{r_{N}^{2}\lrp{\frac{X_{i}(x)}{\sqrt{n}}}>\eta}\leq \frac{1}{\eta}\sum_{i=1}^{n}\Esp{r_{N}^{2}\lrp{\frac{X_{i}(x)}{\sqrt{n}}}}\nonumber\\
	&\leq &\frac{1}{n\eta}\sum_{i=1}^{n}\Esp{r_{N}^{2}\lrp{\varepsilon_{i}}}=\frac{1}{\eta}\Esp{r_{N}^{2}\lrp{\varepsilon_{1}}}
=\frac{1}{\eta}\sum_{l=N}^{\infty}\lambda_{l}(C^{\varepsilon}_{0})\to 0,\quad  N\to \infty,
	\label{sc}\end{eqnarray}

\noindent in view of the trace property of $C^{\varepsilon}_{0}.$

Finally, we prove
\begin{equation}\frac{1}{n}\sum_{1\leq i\leq n}\left\langle X_{i}(x),\phi_{l}\right\rangle_{\mathbb{H}}\left\langle X_{i}(x),\phi_{k}\right\rangle_{\mathbb{H}}
\to_{\mbox{a.s.}} \left\langle C_{x}(\phi_{l}),\phi_{k}\right\rangle_{\mathbb{H}},\quad n\to \infty,\ l,k\geq 1,\label{260}
\end{equation}
\noindent corresponding to  condition (2.60)  in Theorem 2.16 in \cite{Bosq2000}. It is sufficient to prove condition (2.36) in  Corollary 2.3 in \cite{Bosq2000}, which is formulated in terms of
$W_{i}(x)=X_{n+i}(x)\otimes X_{n+i}(x)-C_{0}^{\varepsilon}P_{Y_{0}}\left( E(x)\right),$  $i\in \{0,\dots,p-1\}$. Specifically, under  strictly stationarity of $Y,$ and  $\mathbb{H}$--SWN property of $\varepsilon ,$  we obtain

\begin{eqnarray}
	\Esp{\left\|W_{0}(x)+\dots +W_{p-1}(x)\right\|^{2}_{\mathcal{S}(\mathbb{H})}} &=& \sum_{i,k=0}^{p-1}\Esp{  \left\langle X_{n+i}(x)\otimes X_{n+i}(x),X_{n+k}(x)\otimes X_{n+k}(x)\right\rangle_{\mathcal{S}(\mathbb{H})}}\nonumber \\
	&&\quad -\sum_{i,k=0}^{p-1}\left[P_{Y_{0}}\left( E(x)\right)\right]^{2}\|C_{0}^{\varepsilon}\|_{\mathcal{S}(\mathbb{H})}^{2}\leq  \sum_{i,k=0}^{p-1}\Esp{\inprod{X_{n+i}(x)}{X_{n+k}(x)}^{2}}
	\nonumber  \\ 	
			 &=& p\lrc{\sum_{\substack{u=-(p-1) \\ u\neq 0}}^{p-1}\lrp{1-\frac{|u|}{p}}\Esp{\ind{E^{0}(x)}\ind{E^{u}(x)}
			\sum_{h,l\geq 1}\varepsilon_{1}(\phi_{h})\varepsilon_{1}(\phi_{l})\varepsilon_{u+1}(\phi_{h})
			\varepsilon_{u+1}(\phi_{l})}}\nonumber \\
	&&\quad +p\Esp{\|\varepsilon_{1}\|^{4}_{\mathbb{H}}}P_{Y_{0}}\left( E(x)\right)\leq p\left[\left\|C_{0}^{\varepsilon}\right\|_{\mathcal{S}(\mathbb{H})}^{2}+\Esp{\|\varepsilon_{1}
		\|^{4}_{\mathbb{H}}}\right]<\infty,	 \end{eqnarray}

\noindent where   $\Esp{\|\varepsilon_{1}\|^{4}_{\mathbb{H}}}<\infty$ under
 Assumption A1, and $\left\|C_{0}^{\varepsilon}\right\|_{\mathcal{S}(\mathbb{H})}^{2}<\infty $ under the  trace property of
$C_{0}^{\varepsilon}.$  From  Corollary 2.3 in  \cite{Bosq2000},  with $\gamma =1,$  for all $\beta >1/2,$

$$\frac{n^{1/4}}{(\log(n))^{\beta }}\left\|\frac{S_{n}^{W(x)}}{n}\right\|_{\mathcal{S}(\mathbb{H})}=\frac{n^{1/4}}{(\log(n))^{\beta }}\left\|\sum_{i=1}^{n}\frac{ W_{i}(x)}{n}\right\|_{\mathcal{S}(\mathbb{H})}\to_{\textrm{a.s.}} 0,\ n\to \infty,$$ \noindent   for each  $x\in \mathbb{H},$  and (\ref{260}) holds.

\end{proof}

\subsection{Functional Central Limit Theorem}
\label{s3}

This section provides the convergence in law of $V_{n}$ to time--changed $\mathbb{H}$--valued Wiener process by applying an invariance principle based on Robbins-Monro procedure
 (see Theorem 2 in  \cite{Walk77},  formulated in
 Lemma \ref{th2} of the Appendix).

\begin{theorem}
	\label{th3}
	Under Assumptions A1--A2,  as $n\to \infty,$ the empirical   process $V_{n}$ weak converges in distribution, in the space $C_{\mathbb{H}}([0,1]),$ to a  generalized $\mathbb{H}$--valued  Gaussian process $W_{\infty }$ with covariance operator
	\begin{equation}C_{\min(x,y)}^{W_{\infty}}= C_{0}^{\varepsilon}P_{Y_{0}}\left[\prod_{j=1}^{\infty }(-\infty ,\min\{ x(\phi_{j}),y(\phi_{j})\}]\right],\quad \forall x,y\in \mathbb{H},
		\label{covop}
		\end{equation}
\noindent for  a given  orthonormal basis $\left\{ \phi_{j},\ j\geq 1\right\}$ of $\mathbb{H}.$
	\end{theorem}
\begin{proof}
The proof of this result consists of the verification that the time--changed sequence of  random elements in $C_{\mathbb{H}}([0,1]),$ constructed
   in (\ref{fclt}) below from $\mathbb{H}$--valued empirical process $V_{n},$ satisfies  conditions (i)--(iii) of  Lemma  \ref{th2} in the  Appendix.  Specifically,  (\ref{zerlim}),  (\ref{ic}) and (\ref{i2}) below  lead to the desired result, since the unitary jumps of $V_{n}$ go to zero in probability from  (\ref{convmean}),  and  tightness holds from  (\ref{sc}). A Cram\'er--Wold device argument is applied to prove the convergence  of all the finite--dimensional distributions of $V_{n}$  to a multivariate infinite--dimensional Gaussian probability distribution with  covariance matrix operator having functional entries   \begin{eqnarray}
	C^{X_{i}(x),X_{k}(y)}_{i,k}&:=&\mathbb{E}\left[ X_{i}(x)\otimes X_{k}(y)\right]=\delta_{i,k}C_{0}^{\varepsilon}P\left( \omega \in \Omega ; \ \left\langle Y_{i-1}(\omega ),\phi_{j}\right\rangle_{\mathbb{H}}\leq \min\left(\left\langle  x,\phi_{j}\right\rangle_{\mathbb{H}},\left\langle  y,\phi_{j}\right\rangle_{\mathbb{H}}\right),\ j\geq 1\right)
	\nonumber\\   &=&\delta_{i,k}C_{0}^{\varepsilon}P\left( \omega \in \Omega ; \ \left\langle Y_{0}(\omega ),\phi_{j}\right\rangle_{\mathbb{H}}\leq \min\left(\left\langle  x,\phi_{j}\right\rangle_{\mathbb{H}},\left\langle  y,\phi_{j}\right\rangle_{\mathbb{H}}\right),\ j\geq 1\right)\nonumber\\   &=&
	\delta_{i,k}C_{0}^{\varepsilon}P_{Y_{0}}\left[\prod_{j=1}^{\infty }(-\infty ,\min\{ x(\phi_{j}),y(\phi_{j})\}]\right],\ i,k\geq 1,
	\label{sttb}
\end{eqnarray} (see Appendix).   Thus, the weak convergence in the space $C_{\mathbb{H}}([0,1])$ of $V_{n}$ to zero--mean time--changed $\mathbb{H}$--valued Wiener process $W_{\infty},$ with
   covariance operator
$$\hspace*{-0.2cm}C_{\min(x,y)}^{W_{\infty}}= \mathbb{E}\left[W_{\infty}(P_{Y_{0}}^{-1}(t))W_{\infty}(P_{Y_{0}}^{-1}(s))\right]=C_{0}^{\varepsilon}P_{Y_{0}}\left[\prod_{j=1}^{\infty }(-\infty ,\min\{ x(\phi_{j}),y(\phi_{j})\}]\right],$$
 \noindent holds, with $x=P_{Y_{0}}^{-1}(t)\in \mathbb{H},$ and $y=P_{Y_{0}}^{-1}(s)\in \mathbb{H},$ for any  $t,s\in [0,1].$

Let  $\left\{Y_{n}(t),\ t\in [0,1],\ \ n\geq 1\right\}$ be the time--changed sequence of  random elements in $C_{\mathbb{H}}([0,1])$ defined by
\begin{eqnarray}&&\hspace*{-1cm} Y_{n}(t)=\frac{\sqrt{[nt]}}{\sqrt{n}}	V_{[nt]}(P_{Y_{0}}^{-1}(t))+\frac{(nt-[nt])}{\sqrt{n}}X_{[nt]+1}(P_{Y_{0}}^{-1}(t)),
	\label{fclt}
\end{eqnarray}
\noindent
for each  $t\in [0,1],$ where time--changed is introduced by inverse subordinator $P_{Y_{0}}^{-1}(t)=x,$ $x\in \mathbb{H},$ and hence, $P_{Y_{0}}(x)=t,$ $t\in  [0,1]$ (see  (\ref{eqmin})--(\ref{eqinf})). As before, $V_{n}$ denotes the empirical process   in  (\ref{efp}), and $\{X_{i}(x),\ i\geq 1\}$ is the $\mathbb{H}$--valued martingale difference sequence in Lemma  \ref{lem1a}.

  From (\ref{eq311b2000}), for  $t=P_{Y_{0}}(x)\in [0,1],$
$x\in \mathbb{H}$ (see also  \ref{sop})

\begin{eqnarray}
		&&\hspace*{-1.5cm}\Esp{\left\|\frac{1}{n}\sum_{j=1}^{[nt]}\Esp{X_{j}(P_{Y_{0}}^{-1}(t))\otimes X_{j}(P_{Y_{0}}^{-1}(t))|Y_{j-1}}-P_{Y_{0}}(P_{Y_{0}}^{-1}(t))C_{0}^{\varepsilon}
		\right\|_{L^{1}(\mathbb{H})}} 	\nonumber \\
		&=&  \frac{1}{n}\sum_{j=1}^{[nt]}\mathbb{E}_{Y_{0}}\lrc{\ind{Y_{0}\in E(P_{Y_{0}}^{-1}(t))}^{2}}
	\sum_{l=1}^{\infty }\Esp{\varepsilon_{j}\otimes \varepsilon_{j}(\widetilde{\phi}_{l})(\widetilde{\phi}_{l})}  -\sum_{l=1}^{\infty }P_{Y_{0}}((P_{Y_{0}}^{-1}(t)))C_{0}^{\varepsilon}(\widetilde{\phi}_{l})(\widetilde{\phi}_{l}) \nonumber\\
	&\leq & t\lrc{\sum_{l=1}^{\infty} C_{0}^{\varepsilon}(\widetilde{\phi}_{l})(\widetilde{\phi}_{l})-C_{0}^{\varepsilon}(\widetilde{\phi}_{l})(\widetilde{\phi}_{l})}\leq \lrc{\sum_{l=1}^{\infty} C_{0}^{\varepsilon}(\widetilde{\phi}_{l})(\widetilde{\phi}_{l})-C_{0}^{\varepsilon}(\widetilde{\phi}_{l})(\widetilde{\phi}_{l})}=0, \nonumber\\\label{zerlim}
\end{eqnarray}
\noindent uniformly in $t\in [0,1],$ where, as  in (\ref{KLE}),  $C_{0}^{\varepsilon}(\widetilde{\phi}_{l})(\widetilde{\phi}_{l})=\lambda_{l}(C^{\varepsilon}_{0}),$ $l\geq 1.$   In particular, Lemma 1(i) of the Appendix holds for $t=1.$

From \begin{eqnarray}\mathcal{T}(x)&= &\mathbb{E}\left[\|X_{1}(x)\|_{\mathbb{H}}^{2}\right]=
	\mathbb{E}\left[\left\|\varepsilon_{1}\ind{E^{0}(x)}\right\|_{\mathbb{H}}^{2}\right]=\mathbb{E}\left[\mathbb{E}\left[\left\|[Y_{1}-\Gamma (Y_{0})]\ind{E^{0}(x)}\right\|_{\mathbb{H}}^{2}|Y_{0}\right]\right]\nonumber\\
	&= &\int_{\mathbb{H}}\ind{u\in \mathbb{H};\ \left\langle u,\phi_{j}\right\rangle_{\mathbb{H}}\leq \left\langle  x,\phi_{j}\right\rangle_{\mathbb{H}},\ j\geq 1}^{2}\mbox{Var}\left([Y_{1}-\Gamma (Y_{0})]|Y_{0}=u\right)P_{Y_{0}}(du)=\left\|C_{0}^{\varepsilon}\right\|_{L^{1}(\mathbb{H})}
	P\left( E^{0}(x)\right)=\left\|C_{0}^{\varepsilon}\right\|_{L^{1}(\mathbb{H})}P_{Y_{0}}(E(x)),\nonumber\\
	\label{eqmp2}
\end{eqnarray}

\noindent
 applying strictly stationarity and $\mathbb{H}$--SWN property of $\varepsilon,$  for any  $n\geq 1,$ and $t\in (0,1],$

\begin{eqnarray}
	&\left|\frac{1}{n}\sum_{j=1}^{n}\Esp{\|X_{j}(P_{Y_{0}}^{-1}(t))\|^{2}_{\mathbb{H}}}-P_{Y_{0}}(P_{Y_{0}}^{-1}(t))\left\|C_{0}^{\varepsilon}
\right\|_{L^{1}(\mathbb{H})}\right|=0.
\label{ic}
\end{eqnarray}
\noindent In particular,   Lemma 1(ii)  of the Appendix holds for $t=1.$

Finally,   we prove $L^{1}_{\mathbb{H}}(\Omega,\mathcal{A},P)$ uniform convergence in
Lemma  1(iii) of the Appendix. Specifically, for any $r>0,$

\begin{eqnarray}
	 &&\hspace*{-1cm}\Esp{\frac{1}{n}\sum_{j=1}^{n}\Esp{ \|X_{j}(P_{Y_{0}}^{-1}(t))\|_{\mathbb{H}}^{2}\chi(\|X_{j}(P_{Y_{0}}^{-1}(t))\|^{2}_{\mathbb{H}}\geq rj)|Y_{1},\dots, Y_{j-1}}} \nonumber\\ \hspace*{0.5cm} &= & \Esp{\frac{1}{n}\sum_{j=1}^{n}\Esp{\|X_{j}(P_{Y_{0}}^{-1}(t))\|_{\mathbb{H}}^{2}\chi(\|X_{j}(P_{Y_{0}}^{-1}(t))\|^{2}_{\mathbb{H}}\geq rj)|Y_{j-1}}}  \nonumber\\
	\hspace*{0.5cm}  &\leq &\frac{1}{n}\sum_{j=1}^{n}\int_{\mathbb{H}}\ind{Y_{j-1}\in E(P_{Y_{j-1}}^{-1}(t))}(y_{j-1})\int_{\|\varepsilon_{j}\|_{
			\mathbb{H}}^{2}>rj}\|\varepsilon_{j}\|_{\mathbb{H}}^{2}P\lrp{d \varepsilon_{j}|y_{j-1}}dy_{j-1} \nonumber\\ \hspace*{0.5cm} & = & \frac{1}{n}\sum_{j=1}^{n}\int_{\mathbb{H}}\ind{Y_{0}\in E(P_{Y_{0}}^{-1}(t))}(y_{0})dy_{0}\lrc{\int_{\|\varepsilon_{0}\|_{\mathbb{H}}^{2}>rj}
\|\varepsilon_{0}\|_{\mathbb{H}}^{2}P\lrp{d \varepsilon_{0}}} \leq  \frac{1}{n}\sum_{j=1}^{\infty} \lrc{\int_{\|\varepsilon_{0}\|_{\mathbb{H}}^{2}>rj}
\|\varepsilon_{0}\|_{\mathbb{H}}^{2}P\lrp{d \varepsilon_{0}}}.
\label{i2}
\end{eqnarray}

Since $\mathbb{E}\left[\|\varepsilon_{0}\|_{\mathbb{H}}^{2}\right]<\infty, $ the Dominated Convergence Theorem leads to
\begin{equation}\int_{\|\varepsilon_{0}\|_{\mathbb{H}}^{2}>rj}
\|\varepsilon_{0}\|_{\mathbb{H}}^{2}P\lrp{d \varepsilon_{0}}\leq \mathcal{O}(j^{-\beta }),\ \beta>1,\label{dct}
\end{equation}
\noindent where absolutely integrability criterion has been applied. Hence, $$\sum_{j=1}^{\infty} \int_{\|\varepsilon_{0}\|_{\mathbb{H}}^{2}>rj}
\|\varepsilon_{0}\|_{\mathbb{H}}^{2}P\lrp{d \varepsilon_{0}}<\infty.$$\noindent   Equation (\ref{i2}) then converges to zero as $n\to \infty,$ uniformly in $t\in [0,1].$ In particular,  Lemma 1(iii)  holds for $t=1.$

Equations (\ref{zerlim}),  (\ref{ic}) and (\ref{i2}) also   imply, in particular,  the weak convergence,  in   the Skorokhod space $D_{\mathbb{H}}([0,1]),$ of  $W_{n}(t)=\frac{1}{\sqrt{n}}\sum_{i=1}^{[nt]}X_{i}(P_{Y_{0}}^{-1}(t))$ to  an $\mathbb{H}$--valued zero--mean Wiener process $W_{\infty},$  with autocovariance operator  $C_{0}^{\varepsilon}$ of  $W_{\infty}(P_{Y_{0}}^{-1}(1)).$

\end{proof}

\subsection{Consistency}\label{s4}

In the derivation of consistency of GoF test, the  following assumption is considered, which is of technical nature, providing, in particular,  information about the allowed distance between the null and the alternative hypothesis in the Hilbert--Schmidt operator norm.

\medskip

\noindent \textbf{Assumption A3}. $\Gamma -\Gamma_{0}\in \mathcal{S}(\mathbb{H}),$ where $\Gamma $ denotes the autocorrelation operator under  $H_{1},$ i.e., $P_{Y_{0}}\{ y\in \mathbb{H}:\  \Gamma (y) \neq\Gamma_{0}(y)\}>0,$ with $\Gamma_{0}$ being the  autocorrelation operator under $H_{0},$ which is assumed to be known. For each $x\in \mathbb{H},$ we also assume that  $\mathbb{E}\left[\ind{E^{0}(x)}Y_{0}\right]\otimes \mathbb{E}\left[\ind{E^{0}(x)}Y_{0}\right]\in L^{1}(\mathbb{H}).$

\medskip

In the next result, we will denote for  $y,z\in \mathbb{H},$ $$\lambda (y,z):=\mathbb{E}_{H_{1}}\left[ Y_{1}-\Gamma (Y_{0})+z|Y_{0}=y\right],$$
\noindent and we will consider $z=d(x):=[\Gamma-\Gamma_{0}](x),$ for $x\in \mathbb{H}.$
\begin{theorem}
\label{consistency}
Under Assumptions A1--A3, as $n\to \infty,$
\begin{equation}\sup_{x\in \mathbb{H}}\left\|n^{-1/2}\mathcal{D}_{n}(x)-\mathbb{E}_{H_{1}}\left[\lambda (Y_{0},d(Y_{0}))\ind{E^{(0)}(x)}\right]\right\|_{\mathbb{H}}\to_{\mbox{a.s}} 0,\label{eqc23b}
\end{equation}
\noindent where   $\mathbb{E}_{H_{1}}$ means that expectation is computed under the alternative, and
\begin{equation}\mathcal{D}_{n}(x):= \frac{1}{\sqrt{n}} \sum_{i=1}^{n}\lambda (Y_{i-1},d(Y_{i-1}))\ind{E^{i-1}(x)},\ x\in \mathbb{H},
\label{eqstatasc}
\end{equation}

\noindent with $\lambda (Y_{i-1},d(Y_{i-1}))=[\Gamma-\Gamma_{0}](Y_{i-1}),$
for $i\geq 1.$
\end{theorem}
\begin{proof}
The proof of this result  provides an extension to the $\mathbb{H}$--valued context  of the methodological approach adopted  in pp. 217--218 of \cite{KoulStute1999}. It follows from  condition   (2.36) in  Corollary 2.3 in \cite{Bosq2000}.

Under Assumptions A1--A3, to apply  Corollary 2.3 in  \cite{Bosq2000} we consider $$Z_{i}(x)=\lambda (Y_{i-1},d(Y_{i-1}))\ind{E^{i-1}(x)}-\mathbb{E}_{H_{1}}[\ind{E^{0}(x)} \lambda (Y_{0},d(Y_{0}))],\ i\geq 1.$$ \noindent  Applying strictly stationarity of $Y,$ and equation (\ref{eq311b2000}), we obtain

  \begin{eqnarray}
	\hspace*{-1.5cm}\mathbb{E}_{H_{1}}\left[\left\|Z_{1}(x)+\dots +Z_{p}(x)\right\|_{\mathbb{H}}^{2}\right]
		&\leq& p\sum_{u\in \{-p,\dots,p\}}\left[1-\frac{|u|}{p}\right]\left|\mathbb{E}_{H_{1}}\left[ \ind{E^{0}(x)}\ind{E^{u}(x)}\left\langle \lambda (Y_{0},d(Y_{0})),\lambda (Y_{u},d(Y_{u}))\right\rangle_{\mathbb{H}}\right]\right|\nonumber\\
	&\leq& p\sum_{u\in \mathbb{Z}}\left|\mathbb{E}_{H_{1}}\left[ \ind{E^{0}(x)}\ind{E^{u}(x)}\left\langle \lambda (Y_{0},d(Y_{0})),\lambda (Y_{u},d(Y_{u}))\right\rangle_{\mathbb{H}}\right]\right|\nonumber\\
&=&p\sum_{u\in \mathbb{Z}}\left|\sum_{j\geq 1}[\gamma_{j}\left(\Gamma- \Gamma_{0}\right)]^{2}
\mathbb{E}_{P_{Y_{0}}}\left[\mathbb{E}_{H_{1}}\left[ \ind{E^{0}(x)}\ind{E^{u}(x)}[Y_{0}\otimes Y_{u}](\psi_{j})(\psi_{j})|Y_{0}\right]\right]\right|
\nonumber\\
&=&p\sum_{u\in \mathbb{Z}}\left|\sum_{j\geq 1}[\gamma_{j}\left(\Gamma- \Gamma_{0}\right)]^{2}
\mathbb{E}_{P_{Y_{0}}}\left[\ind{E^{0}(x)}P_{\sum_{t=0}^{u-1}\Gamma^{t}\left(\varepsilon_{u-t} \right)}\left( E_{\boldsymbol{\varepsilon}}^{u-1}(x-\Gamma^{u}(Y_{0}))\right)\right.\right.\nonumber\\
&&\hspace*{1cm}\times \left.\left.\left[\sum_{t=0}^{u-1} \mathbb{E}_{H_{1}}\left[ Y_{0}\otimes \Gamma^{t}\left(\varepsilon_{u-t} \right)(\psi_{j})(\psi_{j})|Y_{0}\right]+\Gamma^{u}[Y_{0}\otimes Y_{0}](\psi_{j})(\psi_{j})\right]\right]\right|\nonumber\\
&\leq& p\sum_{u\in \mathbb{Z}}\left|\sum_{j\geq 1}[\gamma_{j}\left(\Gamma- \Gamma_{0}\right)]^{2}
\mathbb{E}_{P_{Y_{0}}}\left[\ind{E^{0}(x)}\Gamma^{u}[Y_{0}\otimes Y_{0}](\psi_{j})(\psi_{j}) \right]\right|\nonumber\\
&\leq& p \left\|\Gamma- \Gamma_{0}\right\|^{2}_{\mathcal{S}(\mathbb{H})}\sum_{u\in \mathbb{Z}}
\left\|\Gamma^{u}\left[C_{0}^{Y^{T(x)}}+\mathbb{E}[Y^{T(x)}]\otimes \mathbb{E}[Y^{T(x)}]\right]\right\|_{L^{1}(\mathbb{H})}\nonumber\\ &\leq&
 p \left\|\Gamma- \Gamma_{0}\right\|^{2}_{\mathcal{S}(\mathbb{H})}\left[\left\|C_{0}^{Y^{T(x)}}\right\|_{L^{1}(\mathbb{H})}+\left\|\mathbb{E}[Y^{T(x)}]\otimes \mathbb{E}[Y^{T(x)}]\right\|_{L^{1}(\mathbb{H})}\right]
 \sum_{u\in \mathbb{Z}}\left\|\Gamma \right\|^{u}_{\mathcal{L}(\mathbb{H})},
 \label{eqc23bb}
 \end{eqnarray}
 \noindent  for every  $x\in \mathbb{H},$ where
$\left\{\gamma_{j}\left(\Gamma- \Gamma_{0}\right),\ j\geq 1\right\}$  denotes the sequence of singular values of operator $\Gamma- \Gamma_{0}.$ Its corresponding   right and left eigenfunction systems are  $\left\{\psi_{j},\ j\geq 1 \right\}$ and  $\left\{\widetilde{\psi}_{j},\ j\geq 1 \right\}.$ Here,  for $x\in \mathbb{H},$ \begin{eqnarray}
&&P_{\sum_{t=0}^{u-1}\Gamma^{t}\left(\varepsilon_{u-t} \right)}\left( E_{\boldsymbol{\varepsilon}}^{u-1}(x-\Gamma^{u}(y_{0}))\right)
=P\left(\omega \in \Omega;\ \sum_{t=0}^{u-1}\Gamma^{t}\left(\varepsilon_{u-t}(\omega )\right)(\phi_{j})\leq x(\phi_{j})-\Gamma^{u}(y_{0})(\phi_{j})\right),\nonumber \end{eqnarray}
\noindent  and $$C_{0}^{Y^{T(x)}}=\mathbb{E}\left[Y_{0}^{T(x)}\otimes Y_{0}^{T(x)}-\mathbb{E}[Y^{T(x)}]\otimes \mathbb{E}[Y^{T(x)}]\right],$$
\noindent denotes the autocovariance operator of the
truncated random variable
$Y_{0}^{T(x)}=\ind{E^{0}(x)}Y_{0}.$
 $\Gamma^{u}$ is the $u$ power of operator $\Gamma,$ which, as usual,  it is understood as $u$th self--composition of $\Gamma,$
 with  $\sum_{u\in \mathbb{Z}}\left\|\Gamma \right\|^{u}_{\mathcal{L}(\mathbb{H})}<\infty,$ since $\left\|\Gamma \right\|_{\mathcal{L}(\mathbb{H})}<1$
 for invertibility of model (\ref{eq1}) (see also Lemma 3.1, p.74 in \cite{Bosq2000} for more general formulations of such a condition).

Note  that
 $$\left\|C_{0}^{Y^{T(x)}}\right\|_{L^{1}(\mathbb{H})}=\mathbb{E}\left[\|Y_{0}^{T(x)}\|_{\mathbb{H}}^{2}\right]\leq \mathbb{E}\left[\|Y_{0}\|_{\mathbb{H}}^{2}\right]
 =\left\|C_{0}^{Y}\right\|_{L^{1}(\mathbb{H})}<\infty,$$

 \noindent and under Assumption A3, $\left\|\mathbb{E}[Y^{T(x)}]\otimes \mathbb{E}[Y^{T(x)}]\right\|_{L^{1}(\mathbb{H})}<\infty.$

From Corollary 2.3 in  \cite{Bosq2000},   equation (\ref{eqc23bb})  leads, as $n\to \infty,$ to the almost surely convergence
\begin{equation}\sup_{x\in \mathbb{H}}\left\|n^{-1/2}\mathcal{D}_{n}(x)-\mathbb{E}_{H_{1}}\left[\lambda (Y_{0},d(Y_{0}))\ind{E^{(0)}(x)}\right]\right\|_{\mathbb{H}}\to_{\mbox{a.s}} 0,
\end{equation}
\noindent   yielding the consistency of the test.
\end{proof}
\begin{remark}
\label{remtest}
Conditionally to  $\mathbf{h}\in \mathbb{H},$ consistency of the test based on $\mathcal{T}(\mathbf{h})$  follows, in a similar way to (\ref{eqc23bb}),  considering, for  $i=1,\dots,p-1,$ \begin{eqnarray}Z_{i}^{\mathbf{h}}(x)&=&\left\langle Z_{i}(x),\mathbf{h} \right\rangle_{\mathbb{H}}=\lambda (Y_{i-1},d(Y_{i-1}))(\mathbf{h})\ind{E^{i-1}(x)} -\mathbb{E}_{H_{1}}[\ind{E^{0}(x)} \lambda (Y_{0},d(Y_{0}))(\mathbf{h})]\nonumber\\
&=&[\Gamma -\Gamma_{0}](Y_{i-1})(\mathbf{h})\ind{E^{i-1}(x)} -\mathbb{E}_{H_{1}}[\ind{E^{0}(x)} [\Gamma -\Gamma_{0}](Y_{0})(\mathbf{h})],\nonumber\end{eqnarray}
\noindent  and

\begin{eqnarray}
	\hspace*{-2.5cm}
\mathbb{E}_{H_{1}}\left[\left[\sum_{i=1}^{p}Z_{i}^{\mathbf{h}}(x)\right]^{2}\right]
		&\leq& p\sum_{u\in \{-p,\dots,p\}}\left[1-\frac{|u|}{p}\right]\left|\mathbb{E}_{H_{1}}\left[Z_{0}^{\mathbf{h}}(x)Z_{u}^{\mathbf{h}}(x)\right]\right|
\nonumber\\
	&\leq& p\sum_{u\in \mathbb{Z}}\left|\mathbb{E}_{H_{1}}\left[\ind{E^{0}(x)} \ind{E^{u}(x)} \lambda (Y_{0},d(Y_{0}))(\mathbf{h})\lambda (Y_{u},d(Y_{u}))(\mathbf{h})\right]\right|\nonumber\\
&\leq&
 p \left\|[\Gamma- \Gamma_{0}](\mathbf{h})\right\|^{2}_{\mathbb{H}}\left\|C_{0}^{Y^{T(x)}}+\mathbb{E}[Y^{T(x)}]\otimes \mathbb{E}[Y^{T(x)}]\right\|_{L^{1}(\mathbb{H})}\sum_{u\in \mathbb{Z}}\left\|\Gamma \right\|^{u}_{\mathcal{L}(\mathbb{H})}.\nonumber\end{eqnarray}
\end{remark}
\section{Asymptotic properties for composite null hypothesis}
\label{s6}

 In this section, we work under the assumption that $\Gamma_{0}$ is unknown, and must be estimated in a consistent way. We adopt the projection estimation approach introduced in  Chapter 8 in \cite{Bosq2000}, where two scenarios are distinguished, respectively corresponding to the cases of  known and unknown  eigenfunctions  $\left\{ \phi_{j},\ j\geq 1\right\}$ of  the autocovariance operator $C_{0}^{Y}$ of $Y.$ Here, we refer to the case of unknown eigenfunctions of operator $C_{0}^{Y}$ of $Y.$ For the case of known eigenfunctions,  see \ref{unknowneig}.

\medskip

Let $Y$ be an AR$\mathbb{H}$(1) process satisfying the following conditions:

\begin{itemize}\item[(i)] The eigenvalues $\{\lambda_{k}(C_{0}^{Y}),\ k\geq 1\}$ of $C_{0}^{Y}$ are such that
$$\lambda_{1}(C_{0}^{Y})> \lambda_{2}(C_{0}^{Y})>\dots  >\lambda_{j}(C_{0}^{Y})>\dots >0,$$
\noindent where, as before,  $C_{0}^{Y}(\phi_{k})=\lambda_{k}(C_{0}^{Y})\phi_{k},$  for every $k\geq 1.$
\item[(ii)]   For every $n\geq 2$ and $k\geq 1,$ consider $\{ \widetilde{\lambda }_{k,n}(\widehat{C}_{0,n}^{Y}),\ k\geq 1\}$ and  $\{ \phi_{k,n},\ k\geq 1\}$   such that
$$\widehat{C}_{0,n}^{Y}(\phi_{k,n})=\frac{1}{n}\sum_{i=1}^{n}Y_{i}\left\langle Y_{i},\phi_{k,n}\right\rangle_{\mathbb{H}}=
\widetilde{\lambda }_{k,n}(\widehat{C}_{0,n}^{Y})\phi_{k,n},\ k\geq 1,$$
\noindent and assume that   $\widetilde{\lambda }_{k,n}(\widehat{C}_{0,n}^{Y})>0$ a.s.
 \end{itemize}

We consider the estimator  $\widetilde{\Gamma}_{n}$ given by
\begin{equation}\widetilde{\Gamma}_{n}(\varphi )=\sum_{l=1}^{k_{n}}\widetilde{\gamma}_{n,l}(\varphi)\phi_{l,n},\quad \varphi\in \mathbb{H}, \ n\geq 2,
\label{est2}
\end{equation}
\noindent  with  $k_{n}\to \infty,$ and  $k_{n}/n\to 0,$ $n\to \infty,$ and
for $n\geq 1,$  $l\geq 1,$  and  $\varphi\in \mathbb{H},$  $$\widetilde{\gamma}_{n,l}(\varphi )=
	\frac{1}{n-1}\sum_{i=1}^{n-1}\sum_{j=1}^{k_{n}}[\widetilde{\lambda }_{j,n}(\widehat{C}_{0,n}^{Y})]^{-1}\left\langle \varphi ,\phi_{j,n}\right\rangle_{\mathbb{H}}
\left\langle  Y_{i},\phi_{j,n}\right\rangle_{\mathbb{H}}
\left\langle Y_{i+1},\phi_{l,n} \right\rangle_{\mathbb{H}}.$$
Assume that the   conditions of Theorems 8.7--8.8 in  \cite{Bosq2000} hold under the null
$$\widetilde{H_{0}}: \Gamma =\Gamma_{0},\ \mbox{for some} \ \Gamma_{0}\in \Theta_{0}.$$
\noindent Hence,  $\widetilde{\Gamma}_{n}$   provides a strongly consistent estimator of $\Gamma $ under the null.

We  consider the plug--in  $\mathbb{H}$--valued empirical process
\begin{eqnarray}
\widetilde{V}_{n}&:=&\left\{\widetilde{V}_{n}(x),\ x\in \mathbb{H}\right\}=\left\{\frac{1}{\sqrt{n}}\sum_{i=1}^{n}(Y_{i}-\widetilde{\Gamma}_{n} (Y_{i-1}))\ind{E^{i-1}(x)},\ x\in \mathbb{H}\right\}.\label{eqtest}
\end{eqnarray}
\noindent The following result derives the asymptotic equivalence in probability of $V_{n}$ and $\widetilde{V}_{n}.$

\begin{theorem}
\label{th4}
Under conditions of Theorems 3.9 and 8.7--8.8 in  \cite{Bosq2000}, the following identity holds:
\begin{eqnarray}
	\sup_{x\in \mathbb{H}}\|\widetilde{V}_{n}(x)-V_{n}(x)\|_{\mathbb{H}}=o_{P}(1),\quad n\to \infty.
	\label{eqconstest}
\end{eqnarray}
\end{theorem}
\noindent \begin{remark}
From  Theorem \ref{th4}, Theorems \ref{th1} and \ref{th3} also characterize the asymptotic behavior of $\widetilde{V}_{n}.$
\end{remark}
\begin{proof} The main ingredients in the proof of this result are the strong consistency of $\widetilde{\Gamma}_{n},$ derived in
Theorems 8.7--8.8 of  \cite{Bosq2000}, and the convergence to zero in probability of $\frac{\sum_{i=1}^{n}Y_{i-1}}{\sqrt{n}},$ obtained from
 Theorem 3.9 of \cite{Bosq2000}.  Both results are applied to obtain the convergence   to zero in probability of $\widetilde{V}_{n}(x)-V_{n}(x)$
 in $\mathbb{H}$--norm,  leading to their asymptotic equivalence in probability.

Applying Cauchy--Schwartz inequality we obtain  for every $x\in \mathbb{H},$
	\begin{eqnarray}&&
	\hspace*{-0.5cm}	P\left[\|\widetilde{V}_{n}(x)-V_{n}(x)\|_{\mathbb{H}}^{2}>\eta \right]
		 		\leq  P\left[\left\|\widetilde{\Gamma}_{n}-\Gamma\right\|_{\mathcal{L}(\mathbb{H})}
		\left\|\frac{\sum_{i=1}^{n}Y_{i-1}}{\sqrt{n}}\right\|_{\mathbb{H}}>\sqrt{\eta}\right].
		\label{eq1proofth4}
	\end{eqnarray}

From Theorems 8.7--8.8   in  \cite{Bosq2000},  as $n\to \infty,$ $\left\|\widetilde{\Gamma}_{n}-\Gamma\right\|_{\mathcal{L}(\mathbb{H})}$  converges in probability to zero.
 From equation (3.35) in Theorem 3.9 of \cite{Bosq2000}, for $\eta >0,$

\begin{equation}
P\left[\left\|\frac{\sum_{i=1}^{n}Y_{i-1}}{\sqrt{n}}\right\|_{\mathbb{H}}>\eta \right]\leq 4\exp\left(-\frac{\sqrt{n}\eta^{2}}{\alpha_{0}+\beta_{0}\eta}\right),\label{ep1}
\end{equation}
\noindent where $\alpha_{0}$ and $\beta_{0}$ depend on $\Gamma ,$ and $P_{\varepsilon_{1}}.$ Hence, $\left\|\frac{\sum_{i=1}^{n}Y_{i-1}}{\sqrt{n}}\right\|_{\mathbb{H}}$ converges to zero in probability when $n\to \infty.$
Therefore,  (\ref{eq1proofth4}) converges to zero as $n\to \infty.$ Thus, $\widetilde{V}_{n}$ and $V_{n}$ are
asymptotically equivalent in probability in the norm of $\mathbb{H}.$	
\end{proof}

\section{Final comments}
\label{fc}
In the context of  functional residual marked empirical processes  indexed by an $\mathbb{H}$--valued covariate,
  GoF asymptotics based on their  random projection  would require the application of a bivariate version of  Theorem 4.1 in  \cite{Cuestalbertos}. Our paper exploits inverse subordination  for dimension reduction, and technical simplification purposes. Furthermore, our AR$\mathbb{H}$(1) setting involves mild conditions,  allowing the  application of Central Limit Theorems,  and Functional Central Limit Theorems for $\mathbb{H}$--valued martingale differences sequences. This fact constitutes an important advantage of our approach leading to GoF asymptotic analysis in infinite--dimensions.  The novelty in our equivalent formulation of the null hypothesis relies on our characterization of a Borel separating class in $\mathbb{H}$ via infinite--dimensional non--degenerate Gaussian  measures with trace autocovariance operator, whose RKHS defines the index set of the separating class. The $l^{2}$ identification of these measures with infinite product tight Gaussian cylindrical measures on $(\mathbb{R}^{\infty},\mathcal{B}(\mathbb{R}^{\infty}))$ allows us to work in a friendly environment, closely connected with finite--dimensional projections and cylindrical sets.  Indeed, our approach is supported by the similar role that these Gaussian measures in infinite--dimensions play in relation to the uniform measure in finite  dimensions.

In a near future, we also incorporate to our   GoF asymptotic analysis, the formulation of ADF bivariate random projected tests, extending currently applied GoF asymptotic methodology in \cite{Cuesta}, for the functional linear model with scalar response, to the context of the functional linear regression model with $\mathbb{H}$--valued response and covariate. We will cover the cases of independent, and weak--dependent  functional data. Another challenging topic to be addressed  is the formulation, and asymptotic analysis of  GoF tests for  stationary LRD functional time series in  the spectral domain, where the  approach presented in \cite{RuizMedina2022} can support  some advances, in the framework of functional time series in  connected and compact two--point homogeneous spaces (see also \cite{Li2020}  for the nonstationary case).

\section*{Acknowledgements}
\noindent This work has been supported in part by projects PID2022--142900NB-I00, PID2024-158017NB-I00 and PID2020-116587GB-I00, financed by
MCIU/AEI/10.13039/501100011033 and by FEDER UE,  and by CEX2020-001105-M \linebreak  MCIN/AEI/ 10.13039/501100011033. It has also been partially supported  by regional grants ED431C 2021/24 and ED431C 2025/03 (Grupos competitivos),  financed by Xunta de Galicia through European Regional Development Funds (ERDF).

\appendix

 \section{Second--order properties of   $X(x) =\{X_{i}(x),\ i\geq 1\},$ $x\in \mathbb{H}$}
\label{sop}
Note that, for each $x\in \mathbb{H},$ the strongly integrability of the marginals of  the $\mathbb{H}$--valued martingale difference sequence   $\{X_{i}(x),\ i\geq 1\}$     allows the computation of the means and conditional means of the elements of this sequence from  their weak counterparts (Section 1.3 in  \cite{Bosq2000}).

Under \textbf{Assumptions A1--A2},
$\mathbb{E}[X_{i}(x)]=0,$ for any $i\geq 1,$ and $x\in \mathbb{H},$ and  $\mathcal{T}(x)=\mbox{Var}(X_{1}(x))$ can be computed as follows:
\begin{eqnarray}\mathcal{T}(x)&= &\mathbb{E}\left[\|X_{1}(x)\|_{\mathbb{H}}^{2}\right]=
	\mathbb{E}\left[\left\|\varepsilon_{1}\ind{E^{0}(x)}\right\|_{\mathbb{H}}^{2}\right]=\mathbb{E}\left[\mathbb{E}\left[\left\|[Y_{1}-\Gamma (Y_{0})]\ind{E^{0}(x)}\right\|_{\mathbb{H}}^{2}|Y_{0}\right]\right]\nonumber\\
	&= &\int_{\mathbb{H}}\ind{u\in \mathbb{H};\ \left\langle u,\phi_{j}\right\rangle_{\mathbb{H}}\leq \left\langle  x,\phi_{j}\right\rangle_{\mathbb{H}},\ j\geq 1}^{2}\mbox{Var}\left([Y_{1}-\Gamma (Y_{0})]|Y_{0}=u\right)P_{Y_{0}}(du)=\left\|C_{0}^{\varepsilon}\right\|_{L^{1}(\mathbb{H})}
	P\left( E^{0}(x)\right)=\left\|C_{0}^{\varepsilon}\right\|_{L^{1}(\mathbb{H})}P_{Y_{0}}(E(x)),\nonumber\\
	\label{eqmp2}
\end{eqnarray}
\noindent where we  have  applied  $\mbox{Var}\left([Y_{1}-\Gamma (Y_{0})]|Y_{0}=u\right)=\mbox{Var}\left(Y_{1}-\Gamma (Y_{0})\right)$   $=\|C_{0}^{\varepsilon}\|_{L^{1}(\mathbb{H})},$  for all $u\in \mathbb{H},$ under \textbf{Assumptions A1--A2}.
As before, $P_{Y_{0}}$ denotes the infinite--dimensional marginal probability  measure induced by $Y_{0},$ and  $E^{0}(x)$ is the event
$E^{0}(x)=\{\omega \in \Omega ; \ \left\langle Y_{0}(\omega ),\phi_{j}\right\rangle_{\mathbb{H}}\leq \left\langle  x,\phi_{j}\right\rangle_{\mathbb{H}},\ j\geq 1\},$  for each  $x\in \mathbb{H}.$ From \begin{eqnarray}
	Y_{i-1}=\sum_{t=0}^{i-2}\Gamma^{t}(\varepsilon_{i-1-t})+\Gamma^{i-1}(Y_{0}),\quad i\geq 1.
	\label{eq311b2000}
\end{eqnarray}\noindent  applying  the strictly stationarity of  $Y,$ for each $x\in \mathbb{H},$ the autocovariance operator  $C_{0}^{X_{i}(x)}:=\mathbb{E}\left[ X_{i}(x)\otimes X_{i}(x)\right]$ of
$\{X_{i}(x),\ i\geq 1\}$ is given by
\begin{eqnarray}
	C_{0}^{X_{i}(x)}&:=&\mathbb{E}\left[ X_{i}(x)\otimes X_{i}(x)\right]=\mathbb{E}\left[\varepsilon_{i}\otimes \varepsilon_{i}\ind{E^{i-1}(x)}^{2}\right]
	=\mathbb{E}\left[ \ind{E^{i-1}(x)}^{2}
	\mathbb{E}\left[\varepsilon_{i}\otimes \varepsilon_{i}|Y_{i-1}\right]\right]=\mathbb{E}\left[\varepsilon_{i}\otimes \varepsilon_{i}\right]
	P\left( E^{i-1}(x)\right)\nonumber\\
	&= &C_{0}^{\varepsilon}P(E^{(i-1)}(x))= C_{0}^{\varepsilon}P_{Y_{0}}(E(x)),\quad \forall  i\geq 1.
	\label{eqmp2b}
\end{eqnarray}

In a similar way,
the covariance operator $C^{X_{i}(x),X_{k}(y)}_{i,k}:=\mathbb{E}\left[ X_{i}(x)\otimes X_{k}(y)\right]$ can be computed  for every $x,y\in \mathbb{H},$
\begin{eqnarray}
	C^{X_{i}(x),X_{k}(y)}_{i,k}&:=&\mathbb{E}\left[ X_{i}(x)\otimes X_{k}(y)\right]=\delta_{i,k}C_{0}^{\varepsilon}P\left( \omega \in \Omega ; \ \left\langle Y_{i-1}(\omega ),\phi_{j}\right\rangle_{\mathbb{H}}\leq \min\left(\left\langle  x,\phi_{j}\right\rangle_{\mathbb{H}},\left\langle  y,\phi_{j}\right\rangle_{\mathbb{H}}\right),\ j\geq 1\right)
	\nonumber\\   &=&\delta_{i,k}C_{0}^{\varepsilon}P\left( \omega \in \Omega ; \ \left\langle Y_{0}(\omega ),\phi_{j}\right\rangle_{\mathbb{H}}\leq \min\left(\left\langle  x,\phi_{j}\right\rangle_{\mathbb{H}},\left\langle  y,\phi_{j}\right\rangle_{\mathbb{H}}\right),\ j\geq 1\right)\nonumber\\   &=&
	\delta_{i,k}C_{0}^{\varepsilon}P_{Y_{0}}\left[\prod_{j=1}^{\infty }(-\infty ,\min\{ x(\phi_{j}),y(\phi_{j})\}]\right],\ i,k\geq 1,
	\label{sttb}
\end{eqnarray}
\noindent where,  for $i,k\in \mathbb{Z},$
$\delta_{i,k}=0$  if $i\neq k,$ and $\delta_{i,k}=1$ if $i=k.$

\section{Invariance principle based on Robbins--Monro procedure}
\label{Monro}
\begin{lemma}
	\label{th2}
	Let $\{X_{n},\ n\in \mathbb{N}\}$ be a martingale difference sequence of $\mathbb{H}$--valued random variables,  with respect to  the filtration  $\mathcal{M}_{0}^{Y}\subset \mathcal{M}_{1}^{Y}\dots \subset \mathcal{M}_{n}^{Y}\subset \dots,$
	satisfying $\Esp{\|X_{n}\|_{\mathbb{H}}^{2}}<\infty.$  Let $S:\mathbb{H}\to \mathbb{H}$ be a symmetric, positive semidefinite,
	trace operator. For each $n\in \mathbb{N},$ denote as $S^{n}$ the autocovariance operator of $X_{n},$ given $Y_{1},\dots,Y_{n-1}.$  That is,
	$$S^{n}=\Esp{X_{n}\otimes X_{n}|Y_{1},\dots,Y_{n-1}}, \quad n\in \mathbb{N}.$$
	
	Assume that
	\begin{itemize}
		\item[(i)] $\Esp{\left\|\frac{1}{n}\sum_{j=1}^{n}S^{j}-S\right\|_{L^{1}(\mathbb{H})}}\to 0,$ $n\to \infty.$
		\item[(ii)] $\frac{1}{n}\sum_{j=1}^{n}\Esp{\|X_{j}\|^{2}_{\mathbb{H}}}\to \mathrm{trace}(S),$ $n\to \infty.$
		\item[(iii)] For $r>0,$  $$\frac{1}{n}\sum_{j=1}^{n}\Esp{ \|X_{j}\|_{\mathbb{H}}^{2}\chi(\|X_{j}\|^{2}_{\mathbb{H}}\geq rj)|Y_{1},\dots, Y_{j-1}}\to_{P} 0,\ n\to \infty.$$
	\end{itemize}
	
	Then, the sequence of random elements $\{Z_{n}\}$  in $C_{\mathbb{H}}([0,1])$ with the supremum norm, which are defined by
	\begin{equation}
		Z_{n}(t)=\frac{1}{\sqrt{n}}\sum_{j=1}^{[nt]}X_{j}+\left(nt-[nt]\right)\frac{1}{\sqrt{n}}X_{[nt]+1},\quad t\in [0,1],
		\label{eqrech01}
	\end{equation}
	\noindent converges in  distribution to a Brownian motion $W$ in $\mathbb{H},$  with $W(0)=0,$ a.s.,  $\Esp{W(1)}=0,$ and covariance operator $S$ of $W(1).$
\end{lemma}

\section{Composite null hypothesis under known eigenfunctions of $C_{0}^{Y}$}
\label{unknowneig}

Let $Y$ be an AR$\mathbb{H}$(1) process satisfying  the ARH(1) equation \begin{equation}Y_{t}=\Gamma (Y_{t-1})+\varepsilon_{t},\quad t\in \mathbb{Z}.\label{eq1}\end{equation}\noindent  and the following conditions:
\begin{itemize}
	\item[(i)] $Y$  is a standard AR$\mathbb{H}$(1) process with $\mathbb{H}$--SWN innovations, and   $E\|Y_{0}\|_{\mathbb{H}}^{4}<\infty.$
	\item[(ii)] The eigenvalues $\{\lambda_{k}(C_{0}^{Y}),\ k\geq 1\}$ of $C_{0}^{Y}$ are strictly positive.
	\item[(iii)] $P\left(\left\langle Y_{0},\phi_{k}\right\rangle_{\mathbb{H}}=0\right)=0,$ for every $k\geq 1.$
\end{itemize}

Under the  conditions assumed in  Lemma 8.1, and Theorems 8.5--8.6 of \cite{Bosq2000}, a consistent estimator $\widehat{\Gamma}_{n}$ of $\Gamma $ is given by
\begin{equation}\widehat{\Gamma}_{n}(\varphi )=\sum_{l=1}^{k_{n}}\gamma_{n,l}(\varphi )\phi_{l},\quad \varphi \in \mathbb{H}, \ n\geq 2,\label{eqestgamma}
\end{equation}

\noindent where $k_{n}\to \infty,$ and  $k_{n}/n\to 0,$ $n\to \infty,$ and
\begin{eqnarray}\gamma_{n,l}(\varphi )&=&\frac {1}{n-1}\sum_{i=1}^{n-1}\sum_{j=1}^{k_{n}}\widehat{\lambda }_{j,n}^{-1}(C_{0}^{Y})\left\langle \varphi ,\phi_{j}\right\rangle_{\mathbb{H}}
	\left\langle  Y_{i},\phi_{j}\right\rangle_{\mathbb{H}}
	\left\langle Y_{i+1},\phi_{l} \right\rangle_{\mathbb{H}}\label{rhocoef}\\
	\widehat{\lambda }_{k,n} (C_{0}^{Y})&=&  \frac{1}{n}\sum_{i=1}^{n}\left(\left\langle Y_{i}, \phi_{k} \right\rangle_{\mathbb{H}}\right)^{2},\  k\geq 1,\  n\geq 2,
	\nonumber
\end{eqnarray}
\noindent with $\left\{ \phi_{j},\ j\geq 1 \right\}$ being, as before, the  eigenfunction system of $C_{0}^{Y}$ which is assumed to be known.

\noindent Note that under (i)--(iii),

$$\|\widehat{\Gamma}_{n}\|_{\mathcal{L}(\mathbb{H})}\leq \|\widehat{C}_{1,n}^{Y}\|_{\mathcal{L}(\mathbb{H})}\max_{1\leq j\leq k_{n}}\widehat{\lambda }_{j,n}^{-1}(C_{0}^{Y}),\quad \widehat{C}_{1,n}^{Y}=\frac{1}{n-1}\sum_{i=1}^{n-1} Y_{i}\otimes Y_{i+1}.$$

\noindent The design of K--S test under this scenario is based on the  generalized $\mathbb{H}$--valued plug--in empirical process
\begin{equation}\widetilde{V}_{n}^{\phi }(x)=
	\frac{1}{\sqrt{n}}\sum_{i=1}^{n}(Y_{i}- \widehat{\Gamma}_{n} (Y_{i-1}))\ind{E^{i-1}(x)}.
	\label{depbb}
\end{equation}
Specifically, we consider the composite null hypothesis
$$\widetilde{H_{0}}: \Gamma =\Gamma_{0},\ \mbox{for some} \ \Gamma_{0}\in \Theta_{0},$$

\noindent where the family of AR$\mathbb{H}$(1) models under the null $\Theta_{0}$ satisfies the  conditions assumed in  Lemma 8.1, and Theorems 8.5--8.6 of \cite{Bosq2000}.

In a similar way to Theorem 4, from the  strong--consistency in the norm of $\mathcal{L}(\mathbb{H})$ of
$\widehat{\Gamma}_{n},$  under the conditions of Theorems 3.9, the asymptotic equivalence in probability of $\left\{\widetilde{V}_{n}^{\phi },\ x\in \mathbb{H}\right\}$ and
$\left\{V_{n}(x),\ x\in \mathbb{H}\right\}$ follows.  Thus, Theorems 1 and 2 also characterize the asymptotic behavior of $\widetilde{V}_{n}^{\phi }$ under the conditions of Theorems 3.9 and 8.5--8.6  in  \cite{Bosq2000}.

\section{Illustration of the performance of GoF}
\label{ill}
This section presents two illustrations of the proposed GoF test in the context of AR$\mathbb{H}$(1) ($\mathbb{H}=L^{2}([0,1])$), and SP$\mathbb{H}$AR(1) models in \ref{illb} and \ref{sph1}, respectively. In the last case  (see, e.g., \cite{CaponeraMarinucci}),
the assumed  invariance property of the kernels defining the autoregression, and covariance operators  of  SP$\mathbb{H}$AR(1) process  leads to an important dimension reduction in the implementation.
Similar results hold under this invariance property in the case of autoregressive processes on compact and connected two point homogeneous spaces (see, e.g., \cite{MaMalyarenko}; \cite{OvalleRM24};  \cite{RuizMD}).

\subsection{Detecting independence and nonlinearities}
\label{illb}

Finite sample performance of the   proposed  GoF test   is now  illustrated under  simple and composite hypothesis.

Let us first consider the simple null hypothesis
\begin{eqnarray}&& H_0: \Gamma=\mathbf{0},\nonumber\\ && H_1: \Gamma\neq\mathbf{0}.\nonumber\end{eqnarray}
In this section, we adopt the numerical approach  in  \cite{GPortugues21,Alvarez25}  for comparative purposes. The  type I empirical error, and the empirical power, based on $R=500$ repetitions, for functional samples sizes  $n=50,100,200,$ are computed.   The critical values of  the test statistics  are approximated from \emph{Fast Bootstrap} based on $B=2000$ bootstrap replicates.

Under a Gaussian scenario, the functional values of an AR$\mathbb{H}$(1) process are generated with support  in the interval $[0,1],$ evaluated at $71$ temporal nodes.  The   innovation process is assumed to be Gaussian $L^{2}([0,1])$--SWN noise. Its  integral  autocovariance operator has exponential kernel given by

\begin{eqnarray}
	&&  C_{0}^{\varepsilon}(u,v)=\mathbb{E}\left[\varepsilon_{t}(u)\otimes \varepsilon_{t}(v)\right]=
	\sigma^{2}_{\varepsilon}\exp\left(\frac{-|u-v|}{\theta}\right)\nonumber\\ &&  u,v= (i-1)/70, \  i=1,\ldots,71, \ \sigma_{\varepsilon}=0.10, \ \theta=0.6.\nonumber
\end{eqnarray}

\noindent A Gaussian  random initial condition $Y_{0}$ is also generated,  independently of the innovation process $\varepsilon$. $Y_{0}$  has  exponential  autocovariance kernel having the same scale parameter values, $\sigma_{\varepsilon}=0.10,  \theta=0.6, $ as  the innovation process  $\varepsilon$. The recursive equation  (\ref{eq1}) is implemented from  the numerical approximation of the integral $\Gamma(Y_{t-1})(u):=\int \Gamma(u,v)Y_{t-1}(v)dv,$
where, under the alternative  $H_1$,  the kernel of the integral autocorrelation operator   $\Gamma$  is given by $\Gamma(u,v)=\frac{0.7}{71}\exp\left(\frac{-(u^2+v^2)}{0.7468}\right),$ for $u,v \in [0,1]$. The constants are ensuring that the process is stationary.
To remove dependence from the random initial condition,  a  burn-in period of  $n_{0}=500$ observations is considered. That is, for each functional sample size $n,$  the above--described recursive approximation of the values $Y_0,Y_1,\dots,Y_{n_{0}},\dots,Y_{n_{0}+n}$ is computed removing the values  $Y_0,\dots,Y_{n_{0}}.$

As indicated, for comparative purposes, the random projection methodology is implemented  from a bivariate infinite--dimensional Gaussian measure $(\gamma_{\varepsilon }, \gamma_{Y})\sim \mu_{[C_{0}^{\varepsilon}, C_{0}^{Y}]}$ with independent components. Thus, we consider the bivariate random projected  empirical process
$$V_{n}(x)=
\frac{1}{\sqrt{n}}\sum_{i=1}^{n}\left\langle Y_{i}-\Gamma (Y_{i-1}),\gamma_{\varepsilon }\right\rangle_{\mathbb{H}}\ind{ \omega \in \Omega;\ \left\langle Y_{i-1}(\omega ),\gamma_{Y}\right\rangle_{\mathbb{H}}\leq \left\langle  x,\gamma_{Y}\right\rangle_{\mathbb{H}}},$$

\noindent where $\left\{\gamma_{\varepsilon }(u),\  u\in [0,1]\right\}$ and $\left\{\gamma_{Y}(u),\ u\in [0,1]\right\}$
are   independent  realizations of zero--mean Gaussian processes  with  autocovariance operators   $C_{0}^{\varepsilon }$ and $C_{0}^{Y},$ respectively. These processes are generated from their truncated, at term $M=5,$  Karhunen--Lo\'eve expansions.

\emph{Fast Bootstrap} is implemented from the equation
\begin{equation}
	V_{n}^{*b}(x)=
	\frac{1}{\sqrt{n}}\sum_{i=1}^{n}\eta_i \left\langle Y_{i}-\Gamma (Y_{i-1}),\gamma_{\varepsilon }\right\rangle_{\mathbb{H}}\ind{ \omega \in \Omega;\ \left\langle Y_{i-1}(\omega ),\gamma_{Y}\right\rangle_{\mathbb{H}}\leq \left\langle  x,\gamma_{Y}\right\rangle_{\mathbb{H}}},\label{eq:FBSH}
\end{equation}
\noindent where $\eta_{i},$ $i=1,\dots,n,$ are independent and identically distributed standard normal random variables. The computation of the $p$--value is obtained for each projection from
$$p_v(\gamma_\epsilon,\gamma_Y)=\#\left\lbrace \max_{x\in\mathbb{H}}|V_N^{*b}(x)|\geq\max_{x\in\mathbb{H}}|V_N(x)| \right\rbrace/B,$$ \noindent where, as commented,   $B=2000$  bootstrap replicates have been considered. The numbers  of projections tested is  $NP= 1, 2, 3, 4, 5, 10 ,15.$ To obtain only one   $p$--value  the   False Discovery Rate is computed (i.e.,  the expected proportion of false positives among the rejected hypotheses).

For the case of composite null hypothesis  $\widetilde{H}_{0},$ the process generated  is an AR$\mathbb{H}$(1) process, that uses the same $\Gamma$ employed to the alternative hypothesis in the case of simple null hypothesis. For the alternative, we have also considered  an AR$\mathbb{H}$(1) process with the same $\Gamma$ but evaluated in the square process, i.e.
$$Y_{t}(u)=\int \Gamma(u,v)\lrp{\frac{Y_{t-1}}{a_{t-1}}}^2(v)dv+\varepsilon_t(u),$$
where the parameter $a_{t-1}\in\{1, 2\}$ avoids that the process escapes to nonstationarity, being $a_{t-1}=2$ only when $\|Y_{t-1}\|>2$. Although, this is a clearly non AR$\mathbb{H}$(1) process, it is pretty close to an AR$\mathbb{H}$(1) due to the scale of $Y_t$.
Now, the statistics in equation (\ref{eq:FBSH}) must change to incorporate the variability due to the estimation of parameter $\Gamma$. So, following \cite{Escanciano06}, $V_{n}^{*b}(x)$ is changed by $\widetilde{V}_{n}^{*b}(x)$ given by

\begin{equation}
	\widetilde{V}_{n}^{*b}(x)=
	\frac{1}{\sqrt{n}}\sum_{i=1}^{n}\eta_i \left\langle Y^{*b}_{i}-\widehat{\Gamma}^{*b}_{n} (Y_{i-1}),\gamma_{\varepsilon }\right\rangle_{\mathbb{H}}\ind{\omega \in \Omega;\ \left\langle Y_{i-1}(\omega ),\gamma_{Y}\right\rangle_{\mathbb{H}}\leq \left\langle  x,\gamma_{Y}\right\rangle_{\mathbb{H}}},\label{eq:FBCH}
\end{equation}
\noindent where $\widehat{\Gamma}^{*b}_{n}$ is the estimation of the functional parameter and $Y^{*b}_{i}=\widehat{\Gamma}^{*b}_{n}(Y_{i-1})+\varepsilon^{*b}_i$ with $\varepsilon^{*b}_i$ being the resampled errors from the estimated model.

\begin{table}[ht]
	\centering
	\begin{tabular}{|lccccccc|}
		\hline
		$NP$ & 1 & 2 & 3 & 4 & 5 & 10 & 15 \\
		\hline
		$n=50$ & 0.050 & 0.054 & 0.044 & 0.038 & 0.034 & 0.034 & 0.042 \\
		$n=100$ & 0.060 & 0.056 & 0.054 & 0.062 & 0.058 & 0.060 & 0.056 \\
		$n=200$ & 0.050 & 0.054 & 0.056 & 0.052 & 0.048 & 0.052 & 0.040 \\
		\hline
	\end{tabular}
	\caption{Simple Hypothesis. Empirical test size based on $R=500$ repetitions. Number of projections by column, and functional sample size by rows.}
	\label{tabH0}
\end{table}

\begin{table}[ht]
	\centering
	\begin{tabular}{|lccccccc|}
		\hline
		$NP$ & 1 & 2 & 3 & 4 & 5 & 10 & 15 \\
		\hline
		$n=50$ & 0.434 & 0.492 & 0.526 & 0.524 & 0.530 & 0.562 & 0.552 \\
		$n=100$ & 0.486 & 0.630 & 0.678 & 0.728 & 0.764 & 0.836 & 0.848 \\
		$n =200$ & 0.650 & 0.808 & 0.896 & 0.922 & 0.932 & 0.984 & 0.994 \\
		\hline
	\end{tabular}
	\caption{Simple Hypothesis. Empirical test  power based on $R=500$ repetitions. Number of projections by column, and functional sample size by rows.}
	\label{tabH1}
\end{table}

\begin{table}[ht]
	\centering
	\begin{tabular}{|lccccccc|}
		\hline
		$NP$ & 1 & 2 & 3 & 4 & 5 & 10 & 15 \\
		\hline
		$n=50$ & 0.054 & 0.034 & 0.034 & 0.048 & 0.056 & 0.038 & 0.032 \\
		$n=100$ & 0.056 & 0.062 & 0.058 & 0.050 & 0.068 & 0.046 & 0.046 \\
		$n=200$ & 0.050 & 0.050 & 0.052 & 0.046 & 0.054 & 0.048 & 0.048 \\
		\hline
	\end{tabular}
	\caption{Composite Hypothesis. Empirical test size based on $R=500$ repetitions. Number of projections by column, and functional sample size by rows.}
	\label{tabHC0}
\end{table}

\begin{table}[ht]
	\centering
	\begin{tabular}{|lccccccc|}
		\hline
		$NP$ & 1 & 2 & 3 & 4 & 5 & 10 & 15 \\
		\hline
		$n=50$ & 0.114 & 0.096 & 0.084 & 0.062 & 0.064 & 0.062 & 0.066 \\
		$n=100$ & 0.152 & 0.182 & 0.196 & 0.204 & 0.214 & 0.232 & 0.264 \\
		$n=200$ & 0.242 & 0.340 & 0.378 & 0.400 & 0.432 & 0.520 & 0.508 \\
		$n=300$ &0.302  & 0.438   & 0.528  & 0.562  & 0.584  & 0.666   & 0.706\\
		$n=500$ &  0.448 &  0.610 & 0.680 &  0.756 & 0.806 &  0.892 & 0.938\\
		$n=750$ & 0.474 & 0.682 &  0.774  &  0.830  &  0.888  & 0.970 & 0.992\\
		\hline
	\end{tabular}
	\caption{Composite Hypothesis. Empirical test power based on $R=500$ repetitions. Number of projections by column, and functional sample size by rows.}
	\label{tabHC1}
\end{table}

In Tables~\ref{tabH0} and ~\ref{tabH1}, one can respectively find the   $p$--values and power  approximations computed,  based on  $R=500$ repetitions of the implemented testing procedure,  for functional sample sizes $n=50,100,200.$
The Confidence Interval, based on $R=500$ repetitions, obtained for $p=0.05$ is $[0.031, 0.069].$ Specifically,  Table~\ref{tabH0} shows the rejection probability when  $H_{0}$ is true. One can observe  in Table~\ref{tabH0}  that the computed estimates of the test size are slightly below the true parameter  value  $p=0.05$   for the functional sample size $n=50$.
While, for  functional sample sizes $n=100, 200,$ an improvement is observed in the computed estimates of $p=0.05$ for all number of random projections.  For the functional sample sizes analyzed, no patterns are observed regarding computed estimates and the number of random projections.  The  empirical power of the test, for the same functional sample sizes $n=50,100,200,$ and number of repetitions, are displayed in Table~\ref{tabH1}. As  expected,  the  performance of the test is improved  when  the functional sample size $n$ increases. Tables~\ref{tabHC0} and ~\ref{tabHC1} show sizes and powers for the composite hypothesis. The test seems well calibrated as every element in Table~\ref{tabHC0} is inside the Confidence Interval for $p=0.05$ although, for $n=50,$ the test shows a tendency to be below $p=0.05$ similar to what happens in simple null hypothesis. Given the rate of convergence of the empirical eigenfunctions to the theoretical ones  (see Lemma 4.3 and Theorem 4.5 in \cite{Bosq2000}),  larger  functional sample sizes are required to obtain competitive empirical power values.
In Table  \ref{tabHC1}, we illustrate this fact considering additional  functional sample size values $n= 300, 500, 750.$

\subsection{Spherical functional time series \mbox{\emph{SP}}$\mathbb{H}$\mbox{\emph{AR(1)}}}
\label{sph1}
In this section we analyze the performance of the proposed GoF procedure in the special case where $\mathbb{H}=L^{2}(\mathbb{S}_{d},d\nu ,\mathbb{R})$ is  the space   of real--valued square integrable functions  on  the $d$--dimensional sphere in $\mathbb{R}^{d+1}.$
Here, $d\nu$  denotes the normalized Riemannian measure on $\mathbb{S}_{d}.$ In what follows, we will denote by
$\{S_{k,j}^{d}, \ j=1,\dots, \Lambda  (k,d),\ k\in \mathbb{N}_{0}\}$
the  orthonormal basis of   eigenfunctions of the Laplace--Beltrami operator $\Delta_{d}$ on $L^{2}\left(\mathbb{S}_{d},d\nu , \mathbb{R}\right),$ with $\Lambda (k,d)$ denoting the dimension of the $k$th eigenspace of the Laplace--Beltrami operator.  Here, we consider  $d=2.$

The time--varying random projections of the SP$\mathbb{H}$AR(1) process, with repect to the eigenfunctions of the Laplace Beltrami operator,  have been generated by using the MatLab function \emph{arima}. They are evaluated at  $n$ time instants,  where $n$ denotes the functional sample size.  The  functional values of the  innovation process are also  generated from its time--varying projections with respect to such a basis at  $n$ times.

In the generations of SP$\mathbb{H}$AR(1) model we have considered the spherical harmonics $\{S^{2}_{0,0}, S^{2}_{1,0}, S_{1,1}^{2}, S_{2,1}^{2}, S_{2,2}^{2},$ $S_{3,1}^{2}, S_{3,2}^{2}, S_{3,3}^{2}\},$ which are localized in the first four eigenspaces ($k=4$) of  the Laplace Beltrami operator $\Delta_{2}$  on $L^{2}\left(\mathbb{S}_{2},d\nu , \mathbb{R}\right)$ (see
left--hand--side of Figure \ref{f6app2ex3BS62}).
Under the simple null hypothesis $H_{0}=\Gamma_{0},$  one realization (evaluated at $100$ spherical spatial  nodes) of the SP$\mathbb{H}$AR(1) process projected into the corresponding direct sum  $\bigoplus_{k=0}^{3}\mathcal{H}_{k}$ of Laplace Beltrami eigenspaces   is   displayed at the right--hand--side of Figure \ref{f6app2ex3BS62}. Here, $\mathcal{H}_{k}$ denotes the $k$th eigenspace of the Laplace Beltrami operator on $L^{2}(\mathbb{S}_{2},d\nu ,\mathbb{R}),$ $k\in \mathbb{N}_{0}.$ The eigenvalues of the invariant kernel defining the autoregressive operator $\Gamma_{0}$ under simple  $H_{0},$ associated with the Laplace Beltrami  eigenfunctions,  are given by $\{\lambda_{k-1}(\Gamma_{0})=0.7\left(\frac{k+1}{k} \right)^{-3/2},\ k\geq 1\}.$

\begin{figure}[!h]
	\begin{center}
		\includegraphics[width=6.6cm,height=9cm]{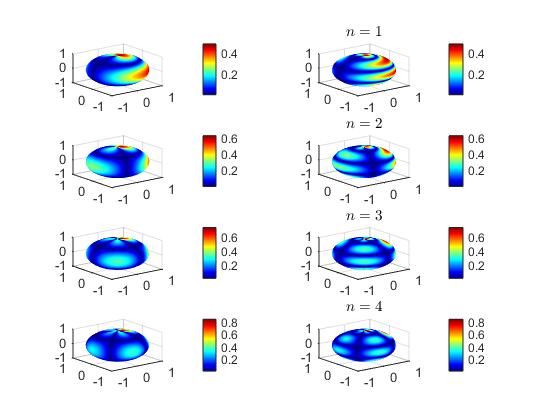}
		\includegraphics[width=6.6cm,height=9cm]{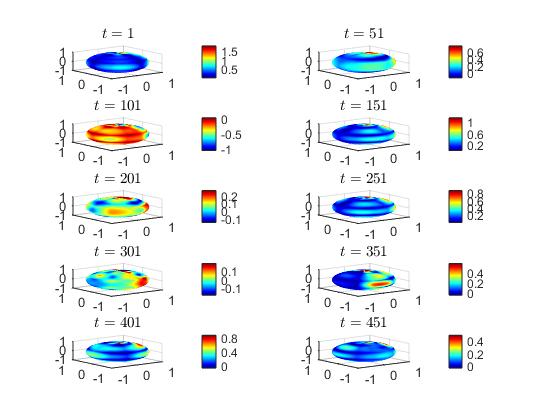} 		
	\end{center}
	\caption{Elements of the truncated orthonormal eigenfunction spherical basis (left--hand--side), and one realization of the functional values of the  generated SP$\mathbb{H}$AR(1) process, projected into  $\bigoplus_{k=0}^{3}\mathcal{H}_{k},$ at times $t=1,51, 101,151, 202,251, 301,351, 401,451,$ from a functional sample of size $n=500$ (right--hand--side)}
	\label{f6app2ex3BS62}
\end{figure}

We first consider  $H_{0}: \Gamma =\Gamma_{0},$ and, under   $H_{1},$ the response is of the form
$Y_{t}^{H_{1}}=[\Gamma_{0}+\widetilde{\Gamma }_{\mathbb{S}_{2}}](Y^{H_{1}}_{t-1})+\varepsilon_{t}^{H_{1}},$
with $\varepsilon_{t}^{H_{1}}=\widetilde{\Gamma }_{\mathbb{S}_{2}}(\varepsilon_{t-1}^{H_{1}})+\eta_{t},$ for  $t\in \mathbb{Z},$  and $\left\{\eta _{t},\ t\in \mathbb{Z}\right\}$  being an $\mathbb{H}$--SWN.
Both operators, $\Gamma_{0}$ and $\widetilde{\Gamma }_{\mathbb{S}_{2}},$ are   invariant bounded linear operators under the group of rotations in the sphere.  The test statistics
\begin{eqnarray}&&\mathcal{T}(\mathbf{h})= \sup_{t\in [0,1]} \left|s_{n}^{-1}(\mathbf{h})\left\langle V_{n}(P_{Y_{0}}^{-1}(t)), \mathbf{h}\right\rangle_{\mathbb{H}}\right|= 	 \sup_{t\in [0,1]} \left|\frac{s_{n}^{-1}(\mathbf{h})}{\sqrt{n}}\sum_{i=1}^{n}\left\langle \varepsilon_{i}(\Gamma_{0}),\mathbf{h}\right\rangle_{\mathbb{H}}\ind{E^{i-1}(P_{Y_{0}}^{-1}(t))}\right|,
	\label{statistic}
\end{eqnarray}
\noindent is evaluated at different  random projections  generated from $\mathbb{H}$--valued Fractional Brownian Motion (FBM) with autocovariance operator $C_{0}^{\varepsilon}.$  The critical value corresponding to $\alpha =0.05$ is $SW_{\alpha}= 2.2414,$ for the supremum norm  of   Brownian motion $W_{t}$ on the interval $[0,1]$.

Under simple null hypothesis (totally specified autocorrelation operator),  Table \ref{tabH0SPH}
displays the empirical test size for the functional sample sizes   $n=50, 100, 200.$ In this case, no  tendency to be below $\alpha =0.05$ is observed
for $n=50.$
In Table \ref{tabH1SPH},  one can find  the empirical test  power values, based on $R=500$ repetitions,  for the functional sample sizes   $n=50, 100, 200, 300, 500, 750.$  As in the previous  $\mathbb{H}=L^{2}([0,1])$ scenario analyzed in Table  \ref{tabH1}, one can observe in Table \ref{tabH1SPH} a better  performance of the test statistics when  the  functional sample size increases.

\begin{table}[ht]
	\centering
	\begin{tabular}{|lccccccc|}
		\hline
		$NP$ & 1 & 2 & 3 & 4 & 5 & 10 & 15 \\
		\hline
		$n=50$ & 0.060 & 0.040  & 0.040 & 0.060 & 0.040 & 0.040 & 0.060 \\
		$n=100$ & 0.040 & 0.040 & 0.040 & 0.060 & 0.060 & 0.060 &  0.040 \\
		$n=200$ & 0.050 & 0.054 & 0.056 & 0.052 & 0.048 & 0.052 & 0.040 \\
		\hline
	\end{tabular}
	\caption{Simple Hypothesis SP$\mathbb{H}$AR(1). Empirical test size based on $R=500$ repetitions. Number of projections by column, and functional sample size by rows.}
	\label{tabH0SPH}
\end{table}

\begin{table}[ht]
	\centering
	\begin{tabular}{|lccccccc|}
		\hline
		$NP$ & 1 & 2 & 3 & 4 & 5 & 10 & 15 \\
		\hline
		$n=50$ & 0.376   & 0.336 &0.718 &0.454 & 0.542 &  0.594&  0.584 \\
		$n=100$ &0.566 & 0.464 & 0.658  & 0.700    & 0.428   &   0.450  & 0.762 \\
		$n=200$  &   0.562 & 0.772 &0.806 & 0.772 &   0.744 &  0.882 & 0.732\\
		$n=300$ & 0.838 & 0.546 & 0.682&  0.944 &   0.864 & 0.912 & 0.892 \\
		$n=500$ & 0.890 & 0.998 & 0.914& 0.994 & 0.978 &   0.984 &   0.984\\	
		$n=750$ & 0.851 & 0.926 & 0.974 & 0.938 & 0.850 & 0.996 & 0.992 \\
		\hline
	\end{tabular}
	\caption{Simple Hypothesis SP$\mathbb{H}$AR(1). Empirical test  power based on $R=500$ repetitions. Number of projections by column, and functional sample size by rows.}
	\label{tabH1SPH}
\end{table}

\begin{table}[ht]
	\centering
	\begin{tabular}{|lccccccc|}
		\hline
		$NP$ & 1 & 2 & 3 & 4 & 5 & 10 & 15 \\
		\hline
		$n=50$ &   0.018 &   0.014 &  0.014   & 0.022       &    0.028 &  0.022 & 0.022\\
		$n=100$ &  0.032  &  0.034  &  0.026  &  0.022 &  0.034  &   0.030 &  0.024\\
		$n=200$ &  0.038 & 0.052 & 0.032 & 0.042 & 0.040 & 0.038 & 0.040 \\
		$n=300$ &
		0.042 &   0.054   &  0.032   & 0.026   & 0.056  &  0.040  & 0.042\\	
		$n=500$ & 0.046 & 0.046 & 0.050 & 0.054 & 0.046 & 0.044 & 0.044 \\
		$n=750$ &0.050 & 0.045 & 0.050 & 0.045 & 0.045&  0.045& 0.055\\
		\hline
	\end{tabular}
	\caption{Composite Hypothesis  SP$\mathbb{H}$AR(1). Empirical test size based on $R=500$ repetitions. Number of projections by column, and functional sample size by rows.}
	\label{tabHC0SPH1}
\end{table}

\begin{table}[ht]
	\centering
	\begin{tabular}{|lccccccc|}
		\hline
		$NP$ & 1 & 2 & 3 & 4 & 5 & 10 & 15 \\
		\hline
		$n=50$ & 0.636 & 0.706 & 0.522 & 0.506 &  0.484 &0.584 &  0.576\\
		$n=100$ &   0.668 & 0.736 &  0.630 & 0.652 & 0.556 & 0.806 & 0.678 \\
		$n=200$ & 0.824 & 0.870 &0.646& 0.912 & 0.612 &  0.062 &  0.638\\
		$n=300$ &  0.962 &0.826  &   0.670 & 0.902  & 0.882  & 0.914 &  0.764 \\
		$n=500$ & 0.994 &  0.966 & 0.998 & 0.978 & 0.986 & 0.918
		&   0.840 \\
		$n=750$ &   0.998     &  0.986  & 0.996 & 0.826 & 0.786 & 0.976   &  0.960 \\
		\hline
	\end{tabular}
	\caption{Composite Hypothesis  SP$\mathbb{H}$AR(1). Empirical test power based on $R=500$ repetitions. Number of projections by column, and functional sample size by rows.}
	\label{tabHC1SPH1}
\end{table}

Under condition (c$_{0}$) in \cite{Bosq2000}  (see p.74, Chapter 3), considering the same autocorrelation operator for the alternative as in  the case of simple null hypothesis, we  test
\begin{eqnarray}
	&&\widetilde{H}_{0}: \|\Gamma_{0}\|_{\mathcal{L}(\mathbb{H})}\leq 1/4, \nonumber\\
	&&\widetilde{H}_{1}: \|\Gamma_{0}\|_{\mathcal{L}(\mathbb{H})}> 1/4.\nonumber\end{eqnarray}
\noindent Since,  under $\widetilde{H}_{1},$  we have considered $\Gamma_{\mathbb{S}_{2}}$ has eigenvalues $\{\lambda_{k-1}(\Gamma_{\mathbb{S}_{2}})=0.5\left(\frac{k+1}{k}\right)^{-3/2}, k \geq 1\},$
$\|\Gamma_{1}\|_{\mathcal{L}(\mathbb{H})}= (0.7+0.5)[1/2]^{3/2}=0.4243>1/4.$
Note that, for the theoretical model, under $\widetilde{H}_{0},$   $\left\{\lambda_{k-1}(\Gamma_{0})= 0.7\left(\frac{k+1}{k} \right)^{-3/2},\ k\geq 1\right\},$  hence, $\|\Gamma_{0}\|_{\mathcal{L}(\mathbb{H})}=0.7[1/2]^{3/2}= 0.2475<1/4.$ But, in practice, we proceed by estimating the  eigenvalues $\left\{\lambda_{k}(\Gamma_{0} ),\ k\in \mathbb{N}_{0}\right\}$ of $\Gamma_{0} $ from the following equation
(see equation (3.13) in  \cite{Bosq2000}):
\begin{eqnarray}
	&&C_{0}^{\varepsilon}=C_{0}^{Y}-\Gamma_{0}C_{0}^{Y}\Gamma_{0}^{\star},
	\label{eqemp}
\end{eqnarray}
\noindent where $\Gamma_{0} ^{\star}$ denotes the adjoint of $\Gamma_{0} ,$ with
$\Gamma_{0}=\Gamma_{0}^{\star}$ in the generations. Specifically, from (\ref{eqemp}),

\begin{equation}\lambda_{k}(\Gamma_{0})=
	\left[1-\lambda_{k}(C_{0}^{\varepsilon})[\lambda_{k}(C_{0}^{Y})]^{-1}
	\right]^{1/2},\quad k\in \mathbb{N}_{0},\label{pps}
\end{equation}

\noindent  which is estimated from the following empirical pure point spectra:
\begin{eqnarray}&&
	\widehat{\lambda}_{k}(C_{0}^{\varepsilon})=\frac{1}{\Lambda  (k,d)}\sum_{j=1}^{\Lambda  (k,d)}
	\frac{1}{n}\sum_{i=1}^{n}\frac{1}{R_{2}}\sum_{l=1}^{R_{2}}\left[\left\langle \varepsilon_{i,l},S_{k,j}^{d}\right\rangle_{L^{2}(\mathbb{S}_{2}, d\nu)}\right]^{2},\ k\in \mathbb{N}_{0},\nonumber \\
	&&\widehat{\lambda_{k}}(C_{0}^{Y})=\frac{1}{\Lambda  (k,d)}\sum_{j=1}^{\Lambda (k,d)}
	\frac{1}{n}\sum_{i=1}^{n}\frac{1}{R_{2}}\sum_{l=1}^{R_{2}}\left[\left\langle Y_{i,l},S_{k,j}^{d}\right\rangle_{L^{2}(\mathbb{S}_{2}, d\nu)}\right]^{2},\ k\in \mathbb{N}_{0},\nonumber \\
	\label{epps}
\end{eqnarray}
\noindent where, for $i=1,\dots,n,$   $\{Y_{i,l},\ l=1,\dots,R_{2}\}$ and $\{\varepsilon_{i,l},\ l=1,\dots,R_{2}\}$ denote the generated  $R_{2}=100$ repetitions  of $Y_{i}$ and  $\varepsilon_{i}$ under $\widetilde{H}_{0},$ respectively, for the estimation of $\Gamma $  under $\widetilde{H}_{0},$ in the SP$\mathbb{H}$AR(1) model.  Lemma 8.1(3) in \cite{Bosq2000}  holds  from equations (\ref{pps})--(\ref{epps}).

Tables~\ref{tabHC0SPH1} and ~\ref{tabHC1SPH1} respectively show the empirical test sizes and powers for  composite null hypothesis under a SP$\mathbb{H}$AR(1) scenario.
In Table~\ref{tabHC0SPH1}, we observe empirical test sizes close to the theoretical value $\alpha =0.05$ for the sample sizes $n=200, 300,500,750.$   The empirical test powers showed in Table~\ref{tabHC1SPH1} are also  computed for the functional sample sizes   $n=50,100,200, 300, 500, 750,$ considering the same number of repetitions $R=500.$  As expected, the empirical test powers are improved when   the functional sample size increases. Under an SP$\mathbb{H}$AR(1) scenario, a better performance in the case of composite  null hypothesis is observed  for small sample sizes $n=50,100, 200,$ since we work under  a lower level of misspecification of the autocorrelation operator $\Gamma $ than in Table \ref{tabHC1} for  $\mathbb{H}=L^{2}([0,1]).$ Under simple $H_{0}$ the reverse  situation is observed in Tables \ref{tabH1}  and \ref{tabH1SPH}. Although, given the variability displayed
by columns in these two tables, for the sample sizes $n=50,100, 200,$ one can not conclude from these numerical results  any significative performance difference between these two testing procedures.

 {\small
\bibliographystyle{myjmva}
\bibliography{GMRMFB}
}
\end{document}